\documentclass{amsart}
\usepackage[utf8]{inputenc}
\usepackage{amsmath}
\usepackage{amsfonts}
\usepackage{amssymb}
\usepackage{amsthm}
\usepackage{amscd}
\usepackage{float}
\usepackage{tikz}
\usepackage{graphicx}
\usepackage[colorlinks=true]{hyperref}
\hypersetup{urlcolor=blue, citecolor=red}
\usepackage{hyperref}

\newtheorem{theorem}{Theorem}[section]
\newtheorem{corollary}[theorem]{Corollary}

\newtheorem{lemma}[theorem]{Lemma}
\newtheorem{proposition}[theorem]{Proposition}

\newtheorem{assumption}[theorem]{Assumption}

\theoremstyle{definition}
\newtheorem{definition}[theorem]{Definition}
\newtheorem{remark}[theorem]{Remark}

\newcommand{\R}{\mathbb{R}}
\newcommand{\N}{\mathbb{N}}
\newcommand{\C}{\mathbb{C}}

\newcommand{\Z}{\mathbb{Z}}

\begin{document}

\title{Square function estimates for cones over quadratic manifolds}
\author{Robert Schippa}
\address[R.~Schippa]{UC Berkeley, Department of Mathematics, 847 Evans Hall
Berkeley, CA 94720-3840}
\email{rschippa@berkeley.edu}

\makeatletter
\@namedef{subjclassname@2020}{%
  \textup{2020} Mathematics Subject Classification}
\makeatother

\begin{abstract}
We extend the $L^4$-square function estimates for the parabola and the half-cone to quadratic manifolds in higher dimensions and their conical extensions. To this end, we require transversality for the tangent spaces of the quadratic manifolds at separated points. This allows us to show biorthogonality for the associated system of quadratic equations. For the conical extensions we obtain a wave envelope estimate by High-Low arguments. 
\end{abstract}


\maketitle

\section{Introduction}
In the following we show $L^4$-square function estimates for quadratic manifolds and their conical extensions. Consider the mapping
\begin{equation}
\label{eq:QuadraticMapping}
Q: \R^d \to \R^l, \quad \xi \mapsto (q_1(\xi),\ldots,q_l(\xi))
\end{equation}
with $q_k(\xi)= \frac{1}{2} \langle \xi, A_k \xi \rangle + \langle b_k, \xi \rangle + c_k$ denoting (affine) quadratic forms with $A_k \in \text{Sym}(\R^d)$, $b_k \in \R^d$, and $c_k \in \R$. Denote the ambient dimension with $n = d+l$. The base frequencies $\xi$ are presently taken from the unit ball. 

\subsection{Square function estimates for quadratic manifolds}

The quadratic manifolds are given as graph of $Q$:
\begin{equation}
\label{eq:ParametrizationQuadraticManifold}
M_Q = \{ (\xi, Q(\xi)) \in \R^n : \, \xi \in B_d(0,1) \}.
\end{equation}

The simplest non-trivial instance is described by the parabola $\mathbb{P}_2$, which is the image of the map\footnote{In the following we abuse notation and typically refer to the map and its image without distinction.}
\begin{equation*}
(-1,1) \ni \xi \mapsto (\xi,\xi^2) \in \R^2.
\end{equation*}

\smallskip

We are currently interested in square function estimates for the canonical partition of the $\delta$-neighborhood $\mathcal{N}_{\delta}(M_Q)$. To this end, we decompose the base frequencies into $\delta^{\frac{1}{2}}$-balls, which essentially flattens the neighborhood. We denote the collection of $\delta^{\frac{1}{2}} \times \delta$-caps $\theta$ forming an essentially disjoint cover of the $\delta$-neighborhood with $\Theta_{\delta}$. The precise definition is given in Section \ref{section:Preliminaries}. The classical C\'ordoba--Fefferman square function estimate (first appeared in \cite{Fefferman1973}, see also \cite{Cordoba1979,Cordoba1982}) reads for a function $f \in \mathcal{S}(\R^2)$ with $\text{supp}(\hat{f}) \subseteq \mathcal{N}_{\delta}(\mathbb{P}_2)$:
\begin{equation*}
\big( \int_{\R^2} |f|^4 \big)^{\frac{1}{4}} \lesssim \big( \int_{\R^2} \big( \sum_{\theta \in \Theta_\delta} |f_{\theta}|^2 \big)^2 \big)^{\frac{1}{4}}.
\end{equation*}

\smallskip

A square function estimate for the complex parabola $\C \mathbb{P}^2$ was recorded by Biggs--Brandes--Hughes \cite{BiggsBrandesHughes2022}. The graph mapping is given by $\C \ni z \mapsto (z,z^2) \in \C^2$, which can be regarded as two-manifold in $\R^4$: $(s,t) \mapsto (s,t,s^2 - t^2, 2st)$.
 For a function $f \in \mathcal{S}(\R^4)$ with Fourier support in a $\delta$-neighborhood $\text{supp}(\hat{f}) \subseteq \mathcal{N}_{\delta}(\C \mathbb{P}_2)$ we have:
\begin{equation*}
\big( \int_{\R^4} |f|^4 \big)^{\frac{1}{4}} \lesssim \big( \int_{\R^4} \big( \sum_{\theta \in \Theta_{\delta}} |f_{\theta}|^2 \big)^2 \big)^{\frac{1}{4}}.
\end{equation*}

Moreover, very recently sharp square function estimates for real moment curves\\ $s \mapsto (s,s^2/2,\ldots,s^n/n!)$ were established by Maldague \cite{Maldague2024} in the cubic case and for all higher degrees, including auxiliary conical sets (\emph{$m$th order Taylor cones}), by Guth--Maldague \cite{GuthMaldague2023}.


\smallskip

Presently, we are interested in $L^4$-square function estimates for general quadratic manifolds. The following condition naturally appears for the generators of the quadratic forms. 
\begin{assumption}
\label{ass:SpanCondition}
Let $A_1,\ldots,A_l \in \text{Sym}(\R^d)$ be symmetric $d \times d$-matrices. We assume that there is $c>0$ such that for any $\nu \in \mathbb{S}^{d-1}$ there is a subset of $\{ A_1,\ldots,A_l \}$ such that
\begin{equation}
\label{eq:SpanCondition}
\det (A_{i_1} \nu, \ldots, A_{i_d} \nu) \geq c > 0. 
\end{equation}
\end{assumption}

As a consequence we obtain quantitative transversality of the tangent spaces at different points: Let $\mathbf{t}_i(\xi) = \partial_i (\xi,Q(\xi))$, $i=1,\ldots,d$ denote tangent vectors spanning the tangent space $T_\xi (M_Q)$. We shall see below that
\begin{equation*}
|\mathbf{t}_1(\xi_1) \wedge \ldots \wedge \mathbf{t}_d(\xi_1) \wedge \mathbf{t}_1(\xi_2) \wedge \ldots \wedge \mathbf{t}_d(\xi_2)| \gtrsim_{|\xi_1-\xi_2|} 1.
\end{equation*}

Under the above assumption, we can show the following:
\begin{theorem}[Square function estimates for quadratic manifolds]
\label{thm:SquareFunctionQuadraticManifold}
 Suppose that the quadratic manifold $M_Q$ is given by \eqref{eq:ParametrizationQuadraticManifold} and for the generators $A_i = \partial^2 q_i$ of the quadratic forms Assumption \ref{ass:SpanCondition} holds true. Then, for $f \in \mathcal{S}(\R^{d+l})$ with $\text{supp}(\hat{f}) \subseteq \mathcal{N}_{\delta}(M_Q)$ it holds
\begin{equation}
\label{eq:L4SquareFunctionQuadraticManifold}
\| f \|_{L^4(\R^n)} \lesssim \big\| \big( \sum_{\theta \in \Theta_{\delta}} |f_{\theta}|^2 \big)^{\frac{1}{2}} \big\|_{L^4(\R^n)}.
\end{equation}
\end{theorem}

\begin{remark}
The implicit constant depends on the transversality parameter $c$ in \eqref{eq:SpanCondition} and the spectral radii of the $A_i$.
\end{remark}

The above square function estimate includes the estimate for the real and complex parabola. For further examples we refer to Section \ref{subsection:ExamplesQuadraticManifolds}.

\subsection{Square function estimates for cones over quadratic manifolds}

In the next step we turn to square function estimates for the conical extension of these manifolds denoted by $\mathcal{C} M_Q$. Define the conical extension of $Q$ by
\begin{equation*}
\begin{split}
\mathcal{C} Q : B_d(0,1) \times [1/2,1] &\to \R^{n+1}, \\
 (\xi,h) &\mapsto (\xi, \frac{q_1(\xi)}{h}, \ldots, \frac{q_l(\xi)}{h}, h) = h ( \frac{\xi}{h}, q_1\big( \frac{\xi}{h} \big), \ldots, q_l \big( \frac{\xi}{h} \big), 1).
 \end{split}
\end{equation*}

$\mathcal{C} M_Q$ is given by the graph of $M_Q$:
\begin{equation}
\label{eq:ParametrizationConical}
\mathcal{C} M_Q = \{ h ( \frac{\xi}{h}, q_1\big( \frac{\xi}{h} \big), \ldots, q_l \big( \frac{\xi}{h} \big), 1) : \quad \xi \in B_d(0,1), \; h \in [1/2,1] \}.
\end{equation}

\smallskip

By the High-Low method, Guth--Wang--Zhang \cite{GuthWangZhang2020} proved a sharp square function estimate for the truncated half-cone, which after a linear transformation becomes the conical extension of the parabola:
\begin{equation*}
(\xi_1,\xi_2) \mapsto (\xi_1,\xi_2,\sqrt{\xi_1^2 + \xi_2^2}).
\end{equation*}
The square function estimate for a canonical covering, i.e., the conical extension of a canonical covering of the base curve $(\xi_1,\xi_2,1)$, $\| (\xi_1,\xi_2) \| = 1$, reads as follows for $\text{supp}(\hat{f}) \subseteq \mathcal{N}_{R^{-1}} (\mathcal{C} \mathbb{P}_2)$:
\begin{equation*}
\int_{\R^4} |f|^4 \lesssim_\varepsilon R^{\varepsilon} \int_{\R^4} \big( \sum_{\theta} |f_{\theta}|^2 \big)^2.
\end{equation*}
A square function estimate for the conical extension of the complex parabola has been shown in \cite{Schippa2023Complex}.

\smallskip

Guth--Wang--Zhang \cite{GuthWangZhang2020} used elegant considerations based on incidence geometry to show a new Kakeya estimate, which uses the high symmetry of the expressions involved. In \cite{Schippa2024} a more analytic and local proof was detailed to analyze conical extensions of degenerate curves and the conical extension of the complex parabola. In \cite{GuthMaldague2023} Guth--Maldague applied the High-Low method to obtain sharp square function estimates for the moment curves. Here we point out that the High-Low method extends to cones over higher-dimensional quadratic manifolds satisfying a transversality condition for tangent spaces at separated points. Furthermore, we surmise that the current approach can be combined with the arguments in \cite{GuthMaldague2023} to show sharp square function estimates for manifolds with higher-order tangential spaces satisfying transversality conditions.

\smallskip

For the conical extension of the more general quadratic manifolds considered above we extend the more analytic approach from \cite{Schippa2024}. This way we obtain an $L^4$-square function estimate provided that \eqref{eq:SpanCondition} holds. \eqref{eq:SpanCondition} also crucially comes into play for the key Kakeya estimate.

\begin{theorem}[Square function estimates for conical extensions of quadratic manifolds]
\label{thm:SquareFunctionConical}
Let $M_Q$ be given by \eqref{eq:ParametrizationQuadraticManifold} with Assumption \ref{ass:SpanCondition} holding true and let $\Gamma M_Q$ be given by \eqref{eq:ParametrizationConical}. Then for $f \in \mathcal{S}(\R^{n+1})$ with $\text{supp}(\hat{f}) \subseteq \mathcal{N}_{\delta}(\Gamma M_Q)$ it holds:
\begin{equation*}
\| f \|_{L^4(\R^{n+1})} \lesssim_\varepsilon \delta^{-\varepsilon} \big\| \big( \sum_{\theta \in \Theta_{\delta}} |f_{\theta}|^2 \big)^{\frac{1}{2}} \big\|_{L^4(\R^{n+1})}.
\end{equation*}
\end{theorem}
For $M_Q$ the real parabola the above theorem recovers the result of Guth--Wang--Zhang \cite{GuthWangZhang2020} and for $M_Q$ the complex parabola the above recovers \cite[Theorem~1.3]{Schippa2024}.

\subsection{Local smoothing estimates for multi-parameter averages}

As an application, we show a smoothing estimate for multi-parameter averages with transversality assumptions in Section \ref{section:MultiParameterAverages}. Precisely, we consider
\begin{equation}
\label{eq:TwoParamLocalSmoothing}
\mathcal{A}^{(2)}_t f(x) = \int f(x- t_1 \gamma_1(s^{(1)}) - t_2 \gamma_2(s^{(2)})) \chi(s^{(1)}) \chi(s^{(2)}) ds^{(1)} ds^{(2)}
\end{equation}
with
\begin{equation*}
\gamma_1(s) = (s_1,s_2,s_1^2 - s_2^2), \quad \gamma_2(s) = (s_1,s_2,2s_1 s_2).
\end{equation*}

We show the following extension of the local smoothing for the wave equation in $2+1$-dimensions (see Section \ref{section:MultiParameterAverages} for further explanations), which describes a two parameter average corresponding to the complex paraboloid.

\begin{theorem}[Local smoothing estimates for multi-parameter averages]
\label{thm:MultiParameterSmoothing}
 Let $\mathcal{A}_t^{(2)}$ be given by \eqref{eq:TwoParamLocalSmoothing}. 
Then, the multi-parameter local smoothing estimate holds:
\begin{equation}
\label{eq:MultiParameterLocalSmoothing}
\| \chi_1(t) \mathcal{A}^{(2)}_{t} f \|_{L^4_{t,x}(\R^{5})} \lesssim_\varepsilon \| f \|_{L^4_{-2+\varepsilon}(\R^3)}.
\end{equation}
\end{theorem}

\medskip

\emph{Outline of the paper.}
In Section \ref{section:Preliminaries} we introduce the canonical coverings of $\delta$-neighborhoods of quadratic manifolds and their conical extensions. In Section \ref{section:SquareFunctionQuadraticManifolds} we show Theorem \ref{thm:SquareFunctionQuadraticManifold} relying on biorthogonality and give examples for quadratic manifolds satisfying the transversality assumption. Next, in Section \ref{section:SquareFunctionConical} we show  Theorem \ref{thm:SquareFunctionConical} by wave envelope estimates. Finally, in Section \ref{section:MultiParameterAverages} we give an application of Theorem \ref{thm:SquareFunctionConical} and show the multi-parameter smoothing estimates for the complex paraboloid stated in Theorem \ref{thm:MultiParameterSmoothing}.

\smallskip

\textbf{Basic notation:}

\begin{itemize}
\item For $x,y \in \R^d$ $\langle x, y \rangle = \sum_{i=1}^d x_i y_i$ denotes the Euclidean scalar product and $\| x \| = \langle x, x \rangle^{\frac{1}{2}}$ the Euclidean norm. For $A \subseteq \R^n$, $x \in \R^n$ we denote with $\text{dist}(x,A)= \inf_{y \in A} \| x - y \|$ the distance to the set $A$.
\item $B_d(x,R) = \{ y \in \R^d : \| x - y \| < R \}$ denotes the ball with radius $R$.
\item $\text{Sym}(\R^d)$ denotes the matrices $A \in \R^{d \times d}$ with $A = A^t$.
\item $\mathcal{S}(\R^d)$ denotes the Schwartz functions, i.e., smooth functions, which together with their derivatives decay faster than any polynomial.
\item The $\delta$-neighborhood of $M$, $M \subseteq \R^n$ is denoted with $\mathcal{N}_{\delta}(M) = \{ x \in \R^n : \text{dist}(x,M) < \delta \}$.
\end{itemize}

\section{Preliminaries}
\label{section:Preliminaries}

In the following we introduce the ``canonical" coverings of the $\delta$-neighborhood of $M_Q$ as given in \eqref{eq:ParametrizationQuadraticManifold} and the conical extension \eqref{eq:ParametrizationConical}. Let $\delta = r^{-2}$. By canonical we refer to isotropic rectangles of size $r^{-1}$, on which the $r^{-2}$-neighborhood of a quadratic manifold is essentially flat.

\smallskip

We denote the $r^{-1} \sigma^{-1}$-lattice of base frequencies by $\mathcal{R}(r^{-2} \sigma^{-2}) = B_d(0,1) \cap r^{-1} \sigma^{-1} \Z^d$. The argument of $\mathcal{R}$ describes the thickness of the neighborhood, which is linearized by choosing the base frequencies on the lattice.

 We first describe the $\delta$-neighborhood of $\Gamma_Q(\xi) = (\xi,q_1(\xi),\ldots,q_l(\xi))$. Define the  tangential vectors as
\begin{equation}
\label{eq:TangentialVectors}
\mathbf{t}_i(\eta) = \partial_i \Gamma_Q(\eta) = (\underbrace{0,\ldots,0,1,0,\ldots,0}_{d \text{ components}},\underbrace{\partial_i q_1(\eta), \ldots, \partial_i q_l(\eta)}_{l \text{ components}}).
\end{equation}
The normal vectors $\mathfrak{n}_i(\eta)$ form a basis of $(T_\eta \Gamma_Q)^{\perp}$. A choice of essentially normalized vectors is given by
\begin{equation}
\label{eq:NormalVectors}
\begin{split}
\mathfrak{n}_1(\eta) &= (\underbrace{-\partial_1 q_1(\eta),\ldots,-\partial_d q_1(\eta)}_{d \text{ components}}, \underbrace{ 1, 0, \ldots,0 }_{l \text{ components}} ), \\
\mathfrak{n}_2(\eta) &= (-\partial_1 q_2(\eta), \ldots, - \partial_d q_2(\eta), 0 , 1, 0 , \ldots, 0), \\
&\vdots \\
\mathfrak{n}_l(\eta) &= (-\partial_1 q_l(\eta), \ldots, - \partial_d q_l(\eta), 0, \ldots, 0, 1).
\end{split}
\end{equation}

\smallskip

For $\eta \in \mathcal{R}(r^{-2})$ we define the $\underbrace{r^{-1} \times \ldots \times r^{-1}}_{d \text{ components}} \times \underbrace{r^{-2} \times \ldots \times r^{-2}}_{n-d=l \text{ components}}$-box\footnote{Dependence on $r$ is suppressed as this will be clear from the context.}:
\begin{equation*}
\begin{split}
\theta(\eta) &= \{ \Gamma_Q(\eta) + a_1 \mathbf{t}_1 (\eta) + \ldots + a_d \mathbf{t}_d(\eta) + b_1 \mathfrak{n}_1(\eta) + \ldots + b_l \mathfrak{n}_l(\eta) : \\
&\quad \quad |a_i| \leq D r^{-1}, \; |b_i| \leq D r^{-2} \}.
\end{split}
\end{equation*}
Let
\begin{equation*}
\Theta_{r^{-2}} = \{ \theta(\eta) : \eta \in \mathcal{R}(r^{-2}) \}
\end{equation*}
denote the collection of boxes canonically covering the $r^{-2}$-neighborhood.

\smallskip

We turn to the conical extension, i.e., the manifold given by the image of
\begin{equation*}
\begin{split}
\Gamma_c(\xi,h) &= (\xi, \frac{q_1(\xi)}{h}, \ldots, \frac{q_l(\xi)}{h}, h) = h \cdot (\frac{\xi}{h}, q_1(\frac{\xi}{h}), \ldots, q_l(\frac{\xi}{h}), 1), \\
&\quad \xi \in B_d(0,1), \; h \in [1/2,1].
\end{split}
\end{equation*}
We define the central line as $\mathbf{c}(\xi) = (\xi,Q(\xi),1)$. Overloading notation and with $\mathbf{t}_i$ and $\mathfrak{n}_i$ given like above, using the trivial embedding $\R^n \hookrightarrow \R^{n+1}, \; \xi \mapsto (\xi,0)$, we define the slabs canonically covering $\mathcal{N}_{r^{-2}}(\Gamma_c)$:
\begin{equation*}
\begin{split}
\theta_{\eta,r^{-2}} &= \{ a \mathbf{c}(\eta) + b_1 \mathbf{t}_1(\eta) + \ldots + b_d \mathbf{t}_d(\eta) + c_1 \mathfrak{n}_1(\eta) + \ldots + c_l \mathfrak{n}_l(\eta) : \\
&\quad a \in [1/2,1], \; |b_i| \leq D r^{-1}, \; |c_i| \leq D r^{-2} \} \subseteq \R^{n+1}
\end{split}
\end{equation*}
with $D \geq 1$ denoting a suitable dimensional constant.
Like above let
\begin{equation*}
\Theta_{r^{-2}} = \{ \theta_{\eta, r^{-2}} : \eta \in \mathcal{R}(r^{-2}) \}
\end{equation*}
denote the collection of $\theta$ canoncially covering the $r^{-2}$-neighborhood.

\section{Square function estimates for quadratic manifolds}
\label{section:SquareFunctionQuadraticManifolds}

This section is devoted to the proof of Theorem \ref{thm:SquareFunctionQuadraticManifold}. Secondly, we give examples.

\subsection{Proof of Theorem \ref{thm:SquareFunctionQuadraticManifold}}

As a preliminary recall the following double-mean value theorem, which is immediate from the fundamental theorem of calculus:
\begin{lemma}[Double-mean~value~theorem]
\label{lem:DMVT}
Let $m \in C^2(\R^n;\R)$ and $\xi_1 + \xi_3 = \xi_2 + \xi_4$ for $\xi_i \in \R^n$, $i=1,\ldots,4$. Then it holds
\begin{equation*}
m(\xi_1) + m(\xi_3) - m(\xi_2) - m(\xi_4) = \int_0^1 \int_0^1 \langle \xi_{12}, \partial^2 m(-\xi_1 + \xi_2 + \xi_4 + s \xi_{12} + t \xi_{14}) \xi_{14} \rangle ds dt
\end{equation*}
with $\xi_{ij} = \xi_i - \xi_j$.
\end{lemma}

We turn to the proof of Theorem \ref{thm:SquareFunctionQuadraticManifold}:
\begin{proof}[Proof~of~Theorem~\ref{thm:SquareFunctionQuadraticManifold}]
Write $f = \sum_{\theta \in \Theta_{\delta}} f_{\theta}$ with $f_{\theta}$ denoting a smoothed version of Fourier projection to $\theta$ such that $\hat{f}_{\theta} = \chi_{\theta} \hat{f}$. Here we require that the $(\chi_{\theta})_{\theta \in \Theta_{\delta}} \subseteq C^\infty_c(\R^n)$ comprise a smooth partition of unity of $\mathcal{N}_{\delta}(M_Q)$: $\sum_{\theta \in \Theta_{\delta}} \chi_{\theta} = 1$.

This allows us to write
\begin{equation*}
\int_{\R^n} |f|^4 = \int_{\R^n} \big( \sum_{\theta \in \Theta_{\delta}} f_{\theta} \big)^2 \big( \sum_{\theta \in \Theta_{\delta}} \overline{f}_{\theta} \big)^2.
\end{equation*}
By an application of Plancherel's theorem we find that only $\theta_1,\ldots,\theta_4$ are contributing if there is a solution to the system of equations with $\xi_i \in \pi_d(\theta_i)$, $\xi_i$ a base frequency for $\theta_i$:
\begin{equation*}
\left\{ \begin{array}{cl}
\xi_1 + \xi_3 &= \xi_2 + \xi_4, \\
q_1(\xi_1) + q_1(\xi_3) &= q_1(\xi_2) + q_1(\xi_4) + \mathcal{O}(\delta), \\
&\vdots \\
q_l(\xi_1) + q_l(\xi_3) &= q_l(\xi_2) + q_l(\xi_4) + \mathcal{O}(\delta).
\end{array} \right.
\end{equation*}
We aim to prove that $\{ \theta_1, \theta_3 \} = \{ \theta_2, \theta_4 \}$ or up to finite translation, i.e.
\begin{equation}
\label{eq:EssentialBiorthogonality}
\theta_1 - \theta_2 \subseteq B(0, C \delta^{\frac{1}{2}}) \vee \theta_1 - \theta_4 \subseteq B(0, C \delta^{\frac{1}{2}}).
\end{equation}

We apply Lemma \ref{lem:DMVT} on the quadratic equations to find the conditions
\begin{equation*}
\langle \xi_{12} , A_m \xi_{14} \rangle = \mathcal{O}(\delta)
\end{equation*}
for $m=1,\ldots,l$ with $A_m$ denoting the generators of the quadratic forms $\xi \mapsto \frac{1}{2} \langle \xi, A_m \xi \rangle$ and $\xi_{ij} = \xi_i - \xi_j$.

We argue by contradiction and suppose that $\min(|\xi_{12}|, |\xi_{14}|) \gg \delta^{\frac{1}{2}}$. So, in this case we divide by the modulus of $\xi_{12}$ and $\xi_{14}$ to find the conditions
\begin{equation}
\label{eq:QuantitativeSpanCondition}
| \langle \xi'_{12}, A_m \xi'_{14} \rangle | \ll 1, \quad \xi'_{ij} = \xi_{ij} / \| \xi_{ij} \|.
\end{equation}
By the transversality condition \eqref{eq:SpanCondition} we can find for $\xi'_{14} \in \mathbb{S}^{d-1}$ a subsequence of $A_m$ such that $A_{i_1} \xi'_{14}, \ldots, A_{i_d} \xi'_{14}$ has maximal rank.

Consequently, we can write
\begin{equation*}
\xi'_{12} = \sum_{m=1}^d a_m A_{i_m} \xi'_{14}
\end{equation*}
for $|a_m| \lesssim 1$. More precisely, we obtain a bound on $a_m$ by writing
\begin{equation*}
\begin{split}
|a_1 \det (A_{i_1} \xi_{12}', A_{i_2} \xi_{14}', \ldots, A_{i_d} \xi_d') | &= |\det( \sum_{m=1}^d a_m A_{i_m} \xi_{14}', A_{i_2} \xi_{14}', \ldots, A_{i_d} \xi_{14}')|  \\
 &\leq |\xi_{12}'| |A_{i_2} \xi_{14}'| \ldots |A_{i_d} \xi_{14}'|.
\end{split}
\end{equation*}
This implies
\begin{equation*}
|a_1| \leq \frac{|\xi_{12}'| |A_{i_2} \xi_{14}'| \ldots |A_{i_d} \xi_{14}'|}{|\det(A_{i_1} \xi_{14}', \ldots,A_{i_d} \xi_{14}')|},
\end{equation*}
and more generally,
\begin{equation*}
|a_j| \leq \frac{|\xi_{12'}| \prod_{\substack{k=1, k \neq j}}^d |A_{i_k} \xi'_{14} |} {|\det(A_{i_1} \xi_{14}', \ldots,A_{i_d} \xi_{14}')|}.
\end{equation*}
The estimate depends on bounds for the eigenvalues of $A_{i_j}$ as reflected in the numerator and the joint transversality of $A_{i_j} \xi_{14}'$, as reflected in the denominator.

Then \eqref{eq:QuantitativeSpanCondition} implies
\begin{equation*}
1 = |\langle \xi'_{12}, \xi'_{12} \rangle| = |\sum_{m=1}^d \langle \xi'_{12}, a_m A_{i_m} \xi'_{14} \rangle| \leq \sum_{m=1}^d |a_m|  | \langle \xi'_{12}, A_{i_m} \xi'_{14} \rangle| \ll 1,
\end{equation*}
which is a contradiction. This completes the proof of \eqref{eq:EssentialBiorthogonality}, from which \eqref{eq:L4SquareFunctionQuadraticManifold} follows by applying the Cauchy-Schwarz inequality.
\end{proof}

\subsection{Examples}
\label{subsection:ExamplesQuadraticManifolds}
First, we note that the scalar case $d=l=1$ with $A_1 = 1 \in \R$ recovers the classical C\'ordoba--Fefferman square function estimate. In the complex case we have
\begin{equation*}
A_1 = 
\begin{pmatrix}
1 & 0 \\
0 & -1
\end{pmatrix}, \quad A_2 = 
\begin{pmatrix}
0 & 1 \\
1 & 0
\end{pmatrix}
.
\end{equation*}
By a change of basis, we can equivalently consider
\begin{equation*}
B_1 = 
\begin{pmatrix}
1 & 0 \\
0 & 1
\end{pmatrix}, \quad B_2 = 
\begin{pmatrix}
0 & 1 \\
-1 & 0
\end{pmatrix}
.
\end{equation*}
$B_2$ describes a rotation by an angle of $\pi / 2$, for which reason
\begin{equation*}
\forall \nu \in \mathbb{S}^1: \, |\det (B_1 \nu, B_2 \nu)| = 1.
\end{equation*}
This shows that Assumption \ref{ass:SpanCondition} is satisfied such that Theorem \ref{thm:SquareFunctionQuadraticManifold} recovers the $L^4$-square function estimate for the complex parabola. 

We turn to higher dimensions. Initially, one might guess that for $d=3$, $l=3$ will suffice to satisfy Assumption \ref{ass:SpanCondition}. This is not true:
Indeed, for any symmetric $A_i \in \R^{3 \times 3}$ we have
\begin{equation*}
\det (A_1 \nu, A_2 \nu, A_3 \nu) = 0
\end{equation*}
for some $\nu \in \mathbb{S}^{2}$. If $A_1$ has a non-trivial kernel, then simply choose $\nu \in \text{ker}(A_1)$. If $A_1$ is invertible, then we have
\begin{equation*}
\det(A_1 \nu, A_2 \nu, A_3 \nu) = |A_1| \cdot \det(\nu, A_1^{-1} A_2 \nu, A_1^{-1} A_3 \nu).
\end{equation*}
With $A_1^{-1} A_2 \in \R^{3 \times 3}$ it must have a real eigenpair. Now choose $\nu$ to be an eigenvector of $A_1^{-1} A_2$. Consequently, the above determinant vanishes.

This shows that in three and higher dimensions, we need more quadratic conditions. A possible choice are reflections and inversions: Consider the matrices $A^{(ij)} \in \R^{d \times d}$, $i \neq j$:
\begin{equation*}
(A^{(ij)} v)_k =  \begin{cases}
v_j, \quad &k=i, \\
v_i, \quad &k=j, \\
v_k, \quad &\text{else},
\end{cases}
\end{equation*}
and $A^{(i)} \in \R^{d \times d}$, $i \in \{1,\ldots,d\}$:
\begin{equation*}
(A^{(i)} v)_k = \begin{cases}
-v_k, \quad &k =i, \\
v_k, \quad &\text{else}.
\end{cases}.
\end{equation*}
Then considering the matrices $\{ I_d, A^{(ij)}, A^{(i)} \}$ we find the transversality assumption \eqref{eq:SpanCondition} to hold by elementary geometric considerations.

\section{Square function estimates for the conical extension}
\label{section:SquareFunctionConical}

In this section we prove Theorem \ref{thm:SquareFunctionConical}. Here we extend the approach of Guth--Wang--Zhang \cite{GuthWangZhang2020} to cones over quadratic manifolds, which satisfy Assumption \ref{ass:SpanCondition}. 
The proof consists of two steps. In the first step we establish a wave envelope estimate, which is based on the key Kakeya estimate. Introduce notations for $\tau \in \Theta_{s^{-2}}$, $s \leq r$:
\begin{equation}
\label{eq:PolarSetKakeyaSorting}
U_{\tau,r^{2}} = \text{conv} \big( \bigcup_{\substack{\theta \in \Theta_{r^{-2}}, \\ \theta \subseteq 10 \tau}} \theta^* \big), \quad S_U f  = \big( \sum_{\theta \subseteq 10 \tau} |f_{\theta}|^2 \big)^{\frac{1}{2}} \big\vert_U.
\end{equation}
Below we tile $\R^{n+1}$ with finitely overlapping translates of $U_{\tau,r^2}$, denoted with $U \parallel U_{\tau,r^2}$.
The Kakeya estimate reads as follows:
\begin{proposition}
\label{prop:KakeyaEstimateConical}
Let $r \gg 1$. The following estimate holds:
\begin{equation*}
\int_{\R^{n+1}} \big| \sum_{\theta \in \Theta_{r^{-2}}} |f_{\theta}|^2 \big|^2 \lesssim \sum_{r^{-1} \leq s \leq 1} \sum_{\tau \in \Theta_{s^2}} \sum_{U \parallel U_{\tau,r^2}} |U|^{-1} \| S_U f \|^4_{L^2(U)}.
\end{equation*}
\end{proposition}

This gives rise to a two-scale wave envelope estimate, which will facilitate the induction-on-scales to show the claimed square function estimate.

\medskip

For the key Kakeya estimate we need to understand the Minkowski sum $\tilde{\theta}(\xi) = \theta(\xi) - \theta(\xi)$. Recall that the tangential and normal vectors of $\Gamma_Q(\xi) = (\xi,Q(\xi))$ are defined in \eqref{eq:TangentialVectors} and \eqref{eq:NormalVectors} and are presently regarded as vectors in $\R^{n+1}$.

\smallskip

 To analyze the overlap of $\tilde{\theta}(\eta)$, for $\sigma \in [r^{-1},1] \cap 2^{\Z}$, $\eta \in \mathcal{R}(r^{-2} \sigma^{-2})$ define the centered plank
\begin{equation*}
\begin{split}
\Theta(\sigma,\eta) &= \{ a \mathbf{c}(\eta) + b_1 \mathbf{t}_1(\eta) + \ldots + b_d \mathbf{t}_d(\eta) + c_1 \mathbf{n}_1(\eta) + \ldots + c_l \mathbf{n}_l(\eta) : \\ 
&\qquad -\sigma^2 \leq a \leq \sigma^2, \; b_i,c_i \in \R, \; |b_i| \leq E r^{-1} \sigma, \; |c_i| \leq E r^{-2} \}.
\end{split}
\end{equation*}

Note that the union forms a canonical covering of the $r^{-1} \sigma$-neighborhood of $\mathcal{C} \Gamma$ for $|a| \sim \sigma^2$, i.e., for the height $\{|\omega_{n+1}| \sim \sigma^2 \}$. The collection of centered planks is denoted by
\begin{equation*}
\mathbf{CP}_{\sigma} = \{ \Theta(\sigma,\eta) : \eta \in \mathcal{R}(r^{-2} \sigma^{-2}) \}
\end{equation*}
and for $\sigma \in [r^{-1},1] \cap 2^{\Z}$ we define
\begin{equation*}
\Omega_{\sigma} = \bigcup \mathbf{CP}_{\sigma} \backslash \bigcup \mathbf{CP}_{\sigma/2}, \text{ and } \Omega_{r^{-1}} = \bigcup \mathbf{CP}_{r^{-1}}.
\end{equation*}
Here we denote the collection of subelements of a set $A$ as
\begin{equation*}
\bigcup A = \{ x : \; \exists y \in A: x \in y \}.
\end{equation*}

\subsection{Kakeya estimates for cones over quadratic manifolds}

The first step in the proof of Proposition \ref{prop:KakeyaEstimateConical} is to invoke Plancherel's theorem:
\begin{equation*}
\int_{\R^{n+1}} \big| \sum_{\theta \in \Theta_{r^{-2}}} |f_{\theta}|^2 \big|^2 = \int_{\R^{n+1}} \big| \sum_{\theta \in \Theta_{r^{-2}}} \widehat{|f_{\theta}|^2}(\omega) \big|^2 d \omega.
\end{equation*}
To analyze the overlap of $\tilde{\theta}$, i.e., the Fourier support of $|f_{\theta}|^2$, we sort them into centered planks subordinate to the partition induced by $\Omega_{\sigma}$.

Indeed, note that the Fourier support of $\sum_{\theta \in \Theta_{r^{-2}}} |f_{\theta}|^2$ is contained in
\begin{equation*}
\bigcup \mathbf{CP}_1 \supseteq \bigcup_{\xi \in \mathcal{R}(r^{-2})} \tilde{\theta}_{\xi}.
\end{equation*}
Secondly, note that
\begin{equation*}
\bigcup \mathbf{CP}_1 = \Omega_1 \cup \Omega_{1/2} \cup \ldots \cup \Omega_{r^{-1}}
\end{equation*}
which allows us to write
\begin{equation}
\label{eq:DecompositionSFFourierSpace}
\int_{\R^{n+1}} \big| \sum_{\theta \in \Theta_{r^{-2}}} \widehat{|f_{\theta}|^2}(\omega) \big|^2 d \omega = \sum_{r^{-1} \leq \sigma \leq 1} \int_{\Omega_{\sigma}} \big| \sum_{\theta \in \Theta_{r^{-2}}} \widehat{|f_{\theta}|^2}(\omega) \big|^2.
\end{equation}

We turn to sorting $\tilde{\theta}$ for $\theta \in \Theta_{r^{-2}}$ into centered planks.


\smallskip

Note that $\omega \in \tilde{\theta}(\eta)$ admits the representation
\begin{equation*}
\omega = a \mathbf{c}(\eta) + \sum_{i=1}^d b_i \mathbf{t}_i(\eta) + \sum_{j=1}^l c_j \mathbf{n}_j(\eta),
\end{equation*}
with $|b_i| \leq D r^{-1}$, $|c_j| \leq D r^{-2}$, $|a| \leq 1$.

Note that the modulus of the last component of $\omega$, $\omega_{n+1}$, is given by $|a|$. Let $\pi_{\leq n}: \R^{n+1} \to \R^n$ denote the projection to the first $n$ components.
For $\omega \in \tilde{\theta}$ we obtain
\begin{equation*}
\pi_{\leq n}(\omega) = h \Gamma_Q(\eta) + \sum_{i=1}^d b_i \mathbf{t}_i(\eta)  + \sum_{j=1}^l c_j \mathbf{n}_j(\eta), \quad |b_i| \lesssim r^{-1} , \quad |c_j| \lesssim r^{-2}.
\end{equation*}

For $\omega \in \Theta(\sigma,\xi)$ we have
\begin{equation*}
\pi_{\leq n}(\omega) = h \Gamma_Q(\xi) + \sum_{i=1}^d \beta_i \mathbf{t}_i(\xi) + \sum_{j=1}^l \gamma_j \mathbf{n}_j(\xi), \quad |\beta_i| \lesssim r^{-1} \sigma, \quad |\gamma_j| \lesssim r^{-2}.
\end{equation*}

We outline the argument: 

\medskip

In $\Omega_{\sigma}$ we sort the $\theta(\xi)$ into centered planks $\Theta(\sigma,\eta)$ by $|\xi - \eta| \leq 4 r^{-1} \sigma^{-1}$.
We shall see that if $\omega \in \tilde{\theta}(\xi) \cap \Omega_{\sigma}$ then we find $\eta \in \mathcal{R}(r^{-2} \sigma^{-2})$ with $|\xi - \eta| \lesssim r^{-1} \sigma^{-1}$ and $\omega \in 10 \Theta(\sigma,\eta)$.
Next, we argue that the $\Theta(\sigma,\eta)$ are essentially disjoint in $\Omega_{\sigma}$. The argument is different for $|h| \sim \sigma^2$ and $|h| \ll \sigma^2$.
\begin{itemize}
\item For $\omega \in \tilde{\theta}\cap \Omega_{\sigma}$ with $|h| \sim \sigma^2$ we find $\pi_{\leq n} \omega$ to be in the $r^{-1} \sigma$-neighborhood of the curve $h \Gamma_Q$. At this point we conclude the finite overlap of $\Theta(\sigma,\xi)$ since they form a canonical covering of the $r^{-1} \sigma$-neighborhood of $h \Gamma_Q$ after projecting to the first $n$ components.
\item Secondly, we shall see that the differencing in $\Omega_{\sigma}$, for $|h| \leq \sigma^2/4$, reduces the $n$-dimensional boxes constituting $\pi_{\leq n}(\Omega_{\sigma})$ (that is $\bigcup \Theta(\sigma,\xi)$ and taking out $\bigcup \Theta(\sigma/2,\xi)$) to its ``ends", where one tangential component has a minimum length  $\sim r^{-1} \sigma$. Then the finite overlap is concluded by the ends being essentially disjoint.
\end{itemize}

The following two lemmas will be helpful. 
The first lemma addresses the differencing when defining $\Omega_{\sigma}$.
\begin{lemma}
\label{lem:RepI}
Let $\sigma \in 2^{\Z} \cap [r^{-1},1]$, $|h| \leq \sigma^2$, and $|\eta| \leq 1$. Let
\begin{equation*}
\omega = h \mathbf{c}(\eta) + \sum_{i=1}^d \ell_i \mathbf{t}_i(\eta) + \sum_{j=1}^l c_j \mathbf{n}_j(\eta), \quad |\ell_i| \leq D r^{-1} \sigma, \; \text{ and } |c_j| \leq D r^{-2}.
\end{equation*}
Then there is $\xi \in \mathcal{R}(r^{-2} \sigma^{-2})$ with $|\xi - \eta| \lesssim r^{-1} \sigma^{-1}$ and
\begin{equation}
\label{eq:RepPII}
\omega = h \mathbf{c}(\xi) + \sum_{i=1}^d \ell'_i \mathbf{t}_i(\xi) + \sum_{j=1}^l c_j' \mathbf{n}_k(\xi), \quad |\ell_i'| \leq C D r^{-1} \sigma, \text{ and } |c_j'| \leq C D r^{-2}.
\end{equation}
For $|h| \ll \sigma^2$ we can choose $\xi \in \mathcal{R}(r^{-2} (\sigma/2)^{-2})$ such that \eqref{eq:RepPII} holds.
\end{lemma}
\begin{proof}
%

Suppose we have the representation
\begin{equation*}
p' = h \Gamma_Q(\eta) + \sum_{i=1}^d \ell_i \mathbf{t}_i(\eta) + \sum_{j=1}^l c_j \mathbf{n}_j(\eta).
\end{equation*}
We use a Taylor expansion to write for $\xi \in \mathcal{R}(r^{-2} \sigma^{-2})$, $\eta = \xi + \Delta \xi$:
\begin{equation*}
\begin{split}
p' &= h \Gamma_Q(\xi + \Delta \xi) + \sum_{i=1}^d \ell_i \mathbf{t}_i(\xi + \Delta \xi) + \sum_{i=1}^l c_j \mathbf{n}_j(\xi + \Delta \xi) \\
&= h (\Gamma_Q(\xi) + \sum_{i=1}^d (\Delta \xi)_i \mathbf{t}_i(\xi) + \sum_{i,j=1}^d \partial^2_{i,j} \Gamma_Q(\xi) (\Delta \xi)_i (\Delta \xi)_j ) \\
&\quad + \sum_{i=1}^d \ell_i ( \mathbf{t}_i(\xi) + \sum_{j=1}^d (\Delta \xi_j) \partial_j \mathbf{t}_i(\xi) ) + \sum_{j=1}^l c_j (\mathbf{n}_j(\xi) + \sum_{i=1}^d (\Delta \xi)_i \partial_i \mathbf{n}_j(\xi)) \\
&= h \Gamma_Q(\xi) + \sum_{i=1}^d (h (\Delta \xi)_i + \ell_i) \mathbf{t}_i(\xi) + \sum_{j=1}^l c_j \mathbf{n}_j(\xi) 
\\ &\quad + \sum_{i,j=1}^d \partial^2 \Gamma_Q(\xi) (\Delta \xi)_i (\Delta \xi)_j h + \sum_{i,j = 1}^d \Delta \xi_j \partial_j \mathbf{t}_i(\xi) + \sum_{i,j} c_j (\Delta \xi)_i \partial_i \mathbf{n}_j(\xi). 
\end{split}
\end{equation*}
For the expressions in the second line note that for $|\Delta \xi| \lesssim r^{-1} \sigma^{-1}$ small such that
\begin{equation}
\label{eq:SizeEstimatesRepresentation}
|h (\Delta \xi)_i + \ell_i| \lesssim r^{-1} \sigma, \quad |\Delta \xi|^2 h \lesssim r^{-2}, \quad |c_j (\Delta \xi)_i| \lesssim r^{-2}.
\end{equation}
It is easy to see that $\partial_j \mathbf{t}_i$, $\partial^2_{i,j} \Gamma_Q$, and $\partial_i \mathbf{n}_j$ can be expressed as linear combinations of $\mathbf{t}_i$ and $\mathbf{n}_j$. This yields the representation
\begin{equation*}
p = h \mathbf{c}(\xi) + \sum_{i=1}^d (h (\Delta \xi)_i + \ell_i + \mathcal{O}(r^{-2})) \mathbf{t}_i(\xi) + \sum_{j=1}^l (c_j+\mathcal{O}(r^{-2})) \mathbf{n}_j(\xi),
\end{equation*}
which completes the proof for $|h| \sim \sigma^2$. For $|h| \ll \sigma^2$, we can choose $\xi \in \mathcal{R}(r^{-2} (\sigma/2)^{-2})$ as slightly larger $|\Delta \xi|$ is tolerable in \eqref{eq:SizeEstimatesRepresentation}. This completes the proof.
\end{proof}

In the second lemma we obtain a representation of $\tilde{\theta}(\xi)$ suitable for future purposes:
\begin{lemma}
\label{lem:RepII}
Let $|h| \leq \sigma^2$, $\omega \in \tilde{\theta}(\xi) \cap \{ \omega_{n+1} = h \} \cap \Omega_{\sigma}$. Then it holds the representation
\begin{equation}
\label{eq:PRepresentation}
\omega = h \mathbf{c}(\eta) + \sum_{i=1}^{d} \ell_i \mathbf{t}_i(\eta) +  \sum_{j=1}^l c_j \mathbf{n}(\eta)
\end{equation}
with $\eta \in \mathcal{R}(r^{-2} \sigma^{-2})$, $|\xi - \eta| \leq 4 r^{-1} \sigma^{-1}$, and $|\ell_i| \lesssim r^{-1} \sigma$, $|c_j| \lesssim r^{-2}$.

\smallskip

For $|h| \ll \sigma^2$, \eqref{eq:PRepresentation} holds with $|\ell| \sim r^{-1} \sigma$.
\end{lemma}
\begin{proof}
First, we turn to the case $|h| \sim \sigma^2$, $\omega \in \tilde{\theta}(\xi) \cap \{ \omega_{n+1} = h \} \cap \Omega_{\sigma}$. This yields the representation
\begin{equation*}
\omega = h \mathbf{c}(\xi)+ \sum_{i=1}^d \ell_i \mathbf{t}_i(\xi) + \sum_{j=1}^l c_j \mathbf{n}_j(\xi) = h \mathbf{c}(\eta) + \sum_{i=1}^d \ell_i' \mathbf{t}_i(\eta) + \sum_{j=1}^l c_j' \mathbf{n}_j(\eta)
\end{equation*}
for $|\ell'| \lesssim r^{-1} \sigma$, $|\ell| \lesssim r^{-1}$, $|c_j|, |c_j'| \lesssim r^{-2}$.

\smallskip

If $|\ell| \lesssim |\ell'| \lesssim r^{-1} \sigma$, we can invoke Lemma \ref{lem:RepI} to conclude the argument. So suppose that $|\ell'| \ll |\ell|$. Let $\eta = \xi + \Delta \xi$. We carry out a Taylor expansion of the right handside at $\xi$:
\begin{equation*}
\begin{split}
&\quad h \mathbf{c}(\xi) + \sum_{i=1}^d \ell_i \mathbf{t}_i(\xi) + \sum_{j=1}^l c_j \mathbf{n}_j(\xi) \\
 &= h \mathbf{c}(\xi + \Delta \xi) + \sum_{i=1}^d \ell_i' \mathbf{t}_i(\xi + \Delta \xi) + \sum_{j=1}^l c_j \mathbf{n}_j(\xi + \Delta \xi) \\
&= h ( \mathbf{c}(\xi) + \sum_{i=1}^d (\Delta \xi)_i \mathbf{t}_i(\xi) + \sum_{i,j=1}^d \frac{(\Delta \xi)_i (\Delta \xi)_j}{2} \partial_j \mathbf{t}_i(\xi)) \\
&\quad + \sum_{i=1}^d \ell_i' \mathbf{t}_i(\xi) + \sum_{i,j=1}^d \ell_i' (\Delta \xi)_j \partial_j \mathbf{t}_i(\xi) + \sum_{m=1}^l c_m \mathbf{n}_m(\xi + \Delta \xi).
\end{split}
\end{equation*}
By subsuming the first $d$ components and then the next $l$ components of the above vector identity, we obtain
\begin{equation*}
\left\{ \begin{array}{cl}
\ell &= h \Delta \xi + \ell' + \mathcal{O}(r^{-2}), \\
\sum_{i,j=1}^d [ h/2 (\Delta \xi)_i (\Delta \xi)_j + \ell_i' (\Delta \xi)_j ] \partial_j \mathbf{t}_i(\xi) &= \mathcal{O}(r^{-2}).
\end{array} \right.
\end{equation*}

The derivatives $\partial_j \mathbf{t}_i$ are comprised of the generators of the quadratic forms:
\begin{equation*}
\partial_j \mathbf{t}_i = 
( 0 , \ldots , 0 , (A_1)_{ij} , \ldots , (A_{l})_{ij})^t.
\end{equation*}
Consequently, the second condition gives
\begin{equation}
\label{eq:SecondConditionII}
\langle \frac{h \Delta \xi}{2} + \ell', A_m (\Delta \xi) \rangle  = \mathcal{O}(r^{-2}), \quad m=1,\ldots,l.
\end{equation}

 Suppose that $|h \Delta \xi| \gg r^{-1} \sigma$. Choose $A_{i_1}, \ldots, A_{i_d}$ such that we have by the transversality assumption \eqref{eq:SpanCondition}:
\begin{equation*}
|\det (A_{i_1} \Delta \xi, \ldots, A_{i_d} \Delta \xi)| \geq c > 0.
\end{equation*}
Then we can choose coefficients $|c_1|,\ldots,|c_d| \lesssim 1$
such that $\Delta \xi = \sum_{m=1}^d c_m A_{i_m} \Delta \xi$. We obtain from taking linear combinations of \eqref{eq:SecondConditionII}
\begin{equation*}
\langle \frac{h \Delta \xi}{2} + \ell', \sum_{m=1}^d c_m A_{i_m} \Delta \xi \rangle = \mathcal{O}(r^{-2}).
\end{equation*}
Since by assumption $|\ell'| \ll h |\Delta \xi|$, it follows
$h |\Delta \xi|^2 = \mathcal{O}(r^{-2})$, which yields $|\Delta \xi| \lesssim r^{-1} \sigma^{-1}$. Consequently,
\begin{equation*}
|\ell| \lesssim h |\Delta \xi| + |\ell'| + \mathcal{O}(r^{-2}) \lesssim r^{-1} \sigma,
\end{equation*}
which completes the proof in case $|h| \sim \sigma^2$.

\medskip

Next, we turn to the case $|h| \ll \sigma^2$. The case $|h| \lesssim r^{-2}$ is straight-forward. In this case a representation with $|c_j|, |c_j'| \lesssim r^{-2}$
\begin{equation*}
\omega = h \mathbf{c}(\xi) + \sum_{i=1}^d \ell_i \mathbf{t}_i(\xi) + \sum_{j=1}^l c_j \mathbf{n}_j(\xi) = h \mathbf{c}(\eta) + \sum_{i=1}^d \ell_i' \mathbf{t}_i(\eta) + \sum_{j=1}^l c'_j \mathbf{n}_j(\eta)
\end{equation*}
implies either $|\ell_i| \sim |\ell_i'| \gg r^{-2}$ or $|\ell| + |\ell'| \lesssim r^{-2}$. Both cases can be concluded readily.

\smallskip

Suppose we have the representation for $\omega \in \tilde{\theta}(\xi) \cap \Omega_{\sigma}$ with
\begin{equation*}
\omega = h \mathbf{c}(\xi) + \sum_{i=1}^d \ell_i \mathbf{t}_i(\xi) + \sum_{j=1}^l c_j \mathbf{n}_j(\xi) = h \mathbf{c}(\eta) + \sum_{i=1}^d \ell_i' \mathbf{t}_i(\eta) + \sum_{j=1}^l c'_j \mathbf{n}_j(\eta)
\end{equation*}
for $|\ell_i| \leq D r^{-1}$ and $|c_j| \leq D r^{-2}$. Note that in case $|\ell| \ll r^{-1} \sigma$ we have $\omega \in \mathbf{CP}_{\sigma/2}$ by Lemma \ref{lem:RepI}, which contradicts $\omega \in \Omega_{\sigma}$.

We rule out the case $|\ell| \gg r^{-1} \sigma \gtrsim |\ell'|$ in the following.

First, suppose that $|\ell'| \lesssim |h|$. In case $|\ell'| \sim |h|$ and $|\ell| \gg |\ell'| \sim |h|$ there is a contradiction taking the modulus of the left- and right-hand side.

If $|\ell'| \ll |h|$ we can write
\begin{equation*}
h \mathbf{c}(\eta) + \sum_{i=1}^d \ell_i' \mathbf{t}_i(\eta) + \sum_{j=1}^l c_j \mathbf{n}_j(\eta) = h \mathbf{c}(\eta + \ell'/h) + \mathcal{O}( \frac{(\ell')^2}{h} + r^{-2}).
\end{equation*}
This shows that $\omega \in \mathcal{N}_{\frac{r^{-1} \sigma}{10}}(h \mathbf{c}(\xi))$, but $\mathbf{CP}_{\sigma/2}$ covers this neighborhood at this height.


\medskip

So, we can suppose that $\frac{|\ell|}{|h|} \gg 1$. Suppose that we have the representation (by our assumption $p \in \tilde{\theta}(\xi) \cap \{ \omega_{n+1} = h \} \cap \Omega_{\sigma}$):
\begin{equation*}
h \mathbf{c}(\xi) + \sum_{i=1}^d \ell'_i \mathbf{t}_i(\xi) + \sum_{i=1}^l c_j' \mathbf{n}_j(\xi) = h \mathbf{c}(\eta) + \sum_{i=1}^d \ell_i \mathbf{t}_i(\eta) + \sum_{j=1}^l c_j \mathbf{n}_j(\eta)
\end{equation*}
for $|\ell'| \gg r^{-1} \sigma$, $|\ell| \lesssim r^{-1} \sigma$, and $|c_j|,|c'_j| \lesssim r^{-2}$. Indeed, if $|\ell'| \lesssim r^{-1} \sigma$, we can finish the proof by invoking Lemma \ref{lem:RepI}. But in case $\frac{|\ell'|}{|h|} \gg \frac{|\ell|}{|h|} \gg 1$ we can divide the above display by $h$ and compute the norms to be
\begin{equation*}
\| (\text{rhs}) \| \sim |\ell'| / |h| \gg  |\ell|/|h| + C r^{-2} / |h|,
\end{equation*}
which yields a contradiction. 

Finally, we point out that for $\omega \in \tilde{\theta}(\xi) \cap \Omega_{\sigma}$ a representation:
\begin{equation}
\label{eq:RepDifferentLength}
\omega = h \mathbf{c}(\xi) + \sum_{i=1}^d \ell_i \mathbf{t}_i(\xi) + \sum_{j=1}^l c_j \mathbf{n}_j(\xi) = h \mathbf{c}(\eta) + \sum_{i=1}^d \ell_i' \mathbf{t}_i(\eta) + \sum_{j=1}^l c_j' \mathbf{n}_j(\eta)
\end{equation}
with $|h| \ll \sigma^2$ and $|\ell_i'| \ll r^{-1} \sigma$ is likewise impossible.

First note that Lemma \ref{lem:RepI} has already given that $|\ell| \sim r^{-1} \sigma$.

In case $|\ell| \gg |h|$, the identity \eqref{eq:RepDifferentLength} is clearly impossible as there is an issue with the norm of the left- and right-hand side.

Suppose that $|\ell'| \ll |\ell| \lesssim |h|$. In this case we can argue like in the preceding paragraph and write
\begin{equation*}
h \mathbf{c}(\eta + \frac{\ell'}{h}) = h \mathbf{c}(\eta) + \sum_{i=1}^d \ell_i' \mathbf{t}_i(\eta) + \mathcal{O}( \frac{(\ell')^2}{h} ),
\end{equation*}
which shows that the right-hand side is in the $\frac{r^{-1} \sigma}{10}$-neighborhood of $h \mathbf{c}(\eta)$, which is contained in $\bigcup \mathbf{CP}_{\sigma/2}$. This makes $\omega \in \Omega_{\sigma}$ impossible.
The proof is complete.

%
%
\end{proof}

The following corollary summarizes up Lemmas \ref{lem:RepI} and \ref{lem:RepII}:
\begin{corollary}
\label{cor:SortingPlanks}
Let $\omega \in \Omega_{\sigma} \cap \tilde{\theta}(\xi)$. Then there is $\Theta(\sigma,\eta) \in \mathbf{CP}_{\sigma}$ with $\omega \in 10 \Theta(\sigma,\eta)$ and $|\xi-\eta| \leq 4 r^{-1} \sigma^{-1}$.
\end{corollary}
The insight is the covering of $\omega$ with a centered plank that is close to the governing direction $\eta$ of $\tilde{\theta}$. This will allow us to sort $\theta$ contributing to $\omega \in \Omega_{\sigma}$ into centered planks of coarser scale.

\medskip

It remains to prove the following lemma regarding the finite overlap of $\Theta(\sigma,\xi)$.
\begin{lemma}
\label{lem:FiniteOverlapCP}
Let $\omega \in \Omega_{\sigma}$. Then the following holds:
\begin{equation*}
\# \{ \Theta(\sigma,\xi) : \omega \in 10 \Theta(\sigma,\xi) \} \lesssim 1.
\end{equation*}
\end{lemma}

The next lemma addresses the finite overlap of the ``ends":
\begin{lemma}
\label{lem:DoubleEndsOverlap}
Let $|h| \ll \sigma^2$ and $\sigma \gg r^{-1}$, and $\mu \in \{1;-1\}$. Let $\omega \in \Omega_{\sigma} \cap \{ \omega_{n+1} = h \}$ and suppose that for $|\ell| \sim |\ell_1| \sim r^{-1} \sigma$, $|C|,|C_1| \lesssim r^{-2}$ and for $\Delta \eta = \eta' - \eta$, $|\Delta \eta| \gg r^{-1} \sigma^{-1}$, it holds
\begin{equation}
\label{eq:DoubleEndsRepresentation}
h \mathbf{c}(\eta) + \sum_{i=1}^d \ell_i \mathbf{t}_i(\eta) + \sum_{j=1}^l C_j \mathbf{n}_j(\eta) = h \mathbf{c}(\eta') + \sum_{i=1}^d \ell_i' \mathbf{t}_i(\eta') + \sum_{j=1}^l C'_j \mathbf{n}_j(\eta').
\end{equation}
Then we have
\begin{equation*}
\ell_i \sim \mu r^{-1} \sigma \Rightarrow \ell'_i \sim - \mu r^{-1} \sigma.
\end{equation*}
\end{lemma}
\begin{proof}
Let $\eta' = \eta + \Delta \eta$. We carry out a Taylor expansion on the right-hand side of the representation:
\begin{equation*}
\begin{split}
h \mathbf{c}(\eta) + \sum_{i=1}^d \ell_i \mathbf{t}_i(\eta) &+ \sum_{j=1}^l C_j \mathbf{n}_j(\eta) = h \mathbf{c}(\eta') + \sum_{i=1}^d \ell_i' \mathbf{t}_i(\eta') + \sum_{j=1}^l C_j' \mathbf{n}_j(\eta') \\
&\; = h \mathbf{c}(\eta) + \sum_{i=1}^d h (\Delta \eta)_i \mathbf{t}_i(\eta) + \sum_{i,j=1}^d \frac{h}{2} (\Delta \eta)_i (\Delta \eta)_j \partial_j \mathbf{t}_i(\eta) \\
&\quad + \sum_{i=1}^d \ell_i' \mathbf{t}_i(\eta) + \sum_{i,j=1}^d \ell_i' (\Delta \eta)_j \partial_j \mathbf{t}_i(\eta) + \mathcal{O}(r^{-2}).
\end{split}
\end{equation*}
We separate the first $d$ components to find the conditions:
\begin{equation}
\label{eq:ConditionsEndsOverlap}
\left\{ \begin{array}{cl}
h \Delta \eta + \ell' &= \ell + \mathcal{O}(r^{-2}), \\
\sum_{i,j=1}^d [ \frac{h}{2} (\Delta \eta)_i (\Delta \eta)_j + \ell_i' (\Delta \eta)_j ] \partial_i \mathbf{t}_j(\eta) &= \mathcal{O}(r^{-2}). 
\end{array} \right.
\end{equation}

The same reasoning as in the proof of Lemma \ref{lem:RepII} shows that by the transversality assumption \eqref{eq:SpanCondition} the second identity leads us to the condition:
\begin{equation*}
\big| \frac{h \Delta \eta}{2} + \ell' \big| |\Delta \eta| = \mathcal{O}(r^{-2}).
\end{equation*}
By the assumption $|\Delta \eta| \gg r^{-1} \sigma^{-1}$, we find
\begin{equation*}
\big| \frac{h \Delta \eta}{2} + \ell' \big| \ll r^{-1} \sigma.
\end{equation*}
Consequently, by the first condition in \eqref{eq:ConditionsEndsOverlap} and $\sigma \gg r^{-1}$,
\begin{equation*}
\big| \frac{h \Delta \eta}{2} + \ell' + \frac{h \Delta \eta}{2} - \ell \big| = \mathcal{O}(r^{-2}) \Rightarrow \big| \frac{h \Delta \eta}{2} - \ell \big| \ll r^{-1} \sigma.
\end{equation*}
The claim follows by projecting to the $i$th component.
\end{proof}

\begin{proof}[Proof~of~Lemma~\ref{lem:FiniteOverlapCP}]
Notably, for $\sigma \lesssim r^{-1}$ we have $\# \mathbf{CP}_{\sigma} = \mathcal{O}(1)$. For $r^{-1} \ll \sigma $ we distinguish between $|h| \sim \sigma^2$ and $|h| \ll \sigma^2$.

We turn to $|h| \ll \sigma^2$ first. In this case the finite overlap of $\Theta(\sigma,\eta)$ follows from $\tilde{\theta}(\xi) \cap \Omega_{\sigma}$ being comprised of the ``ends": As a consequence of Lemma \ref{lem:RepII}
$\omega \in \Omega_{\sigma} \cap \{ \omega_{n+1} = h \} \cap \tilde{\theta}(\xi)$ admits only the representation:
\begin{equation*}
p = h \mathbf{c}(\eta) + \sum_{i=1}^d \ell_i \mathbf{t}_i(\eta) + \sum_{j=1}^l c_j \mathbf{n}_j(\eta)
\end{equation*}
for $|\ell| \sim r^{-1} \sigma$ and $|c_j| \lesssim r^{-2}$. This means that $\Omega_{\sigma}$ consists essentially of the ends of $\Theta(\sigma,\eta)$. Let
\begin{equation*}
\mathcal{M} = \{ \eta \in \mathcal{R}(r^{-2} \sigma^{-2}) : p \in \Theta(\sigma,\eta) \}
\end{equation*}
denote the collection of $\eta$ such that $p \in \Theta(\sigma,\eta)$. For $\eta \in \mathcal{M}$ let
\begin{equation*}
p = h \mathbf{c}(\eta) + \sum_{i=1}^d \ell_i(p,\eta) \mathbf{t}_i(\eta) + \sum_{j=1}^l c_j(p,\eta) \mathbf{n}_j(\eta)
\end{equation*}
denote the unique representation of $p \in \Theta(\sigma,\eta)$.  We sort the collection into
\begin{equation*}
\mathcal{M}_{i,\mu} = \{ \eta \in \mathcal{M} \, : \, \ell_i(p,\eta) \sim \mu r^{-1} \sigma \}
\end{equation*}
with $i=1,\ldots,d$, $\mu \in \{ -1 ; +1 \}$. By
\begin{equation*}
\mathcal{M} = \bigcup_{\substack{i=1,\ldots,d, \\ \mu \in \{-1;+1\}}} \mathcal{M}_{i,\mu}
\end{equation*}
it suffices to show $\# \mathcal{M}_{i,\mu} = \mathcal{O}(1)$ for any $i\in \{1,\ldots,d\}$ and $\mu \in \{\pm 1\}$.

Consider $i,\mu$, for which $\mathcal{M}_{i,\mu} \neq \emptyset$ and $\eta_1,\eta_2 \in \mathcal{M}_{i,\mu}$. The identity
\begin{equation*}
\begin{split}
&\quad h \mathbf{c}(\eta_1) + \sum_{i=1}^d \ell_i(p,\eta_1) \mathbf{t}_i(\eta_1) + \sum_{j=1}^l c_j(p,\eta_1) \mathbf{n}_j(\eta_1) \\
&= h \mathbf{c}(\eta_2) + \sum_{i=1}^d \ell_i(p,\eta_2) \mathbf{t}_i(\eta_2) + \sum_{j=1}^l c_j(p,\eta_2) \mathbf{n}_j(\eta_2)
\end{split}
\end{equation*}
and the fact that $\ell_i(p,\eta_j) \sim \mu r^{-1} \sigma$, $j=1,2$ implies together with Lemma \ref{lem:DoubleEndsOverlap} that $|\eta_1-\eta_2| \lesssim r^{-1} \sigma^{-1}$. Since the distance for distinct $\eta_1,\eta_2 \in \mathcal{R}(r^{-2} \sigma^{-2})$ is comparable to $r^{-1} \sigma^{-1}$ this shows that $\# \mathcal{M}_{i,\mu} = \mathcal{O}(1)$ as claimed. The proof in case $|h| \ll \sigma^2$ is complete.


It remains to check the case $|h| \sim \sigma^2$. With the finite overlap of $\mathbf{CP}_{\sigma}$ for $|h| \sim \sigma^2$, which follows from the centered planks $\mathbf{CP}_{\sigma}$ at height $|h| \sim \sigma^2$ comprising a canonical covering of the $r^{-1} \sigma$-neighborhood of $\{ h \mathbf{c}(\xi) : |\xi| \leq 1 \}$, we can establish Lemma \ref{lem:FiniteOverlapCP}.
\end{proof}

We are ready to complete the proof of Proposition \ref{prop:KakeyaEstimateConical}:

\begin{proof}[Proof~of~Proposition~\ref{prop:KakeyaEstimateConical}]
Decompose the Fourier support of $|f|^2$ into $\Omega_{\sigma}$ like in \eqref{eq:DecompositionSFFourierSpace}.
Next, sort $\theta$ in \eqref{eq:DecompositionSFFourierSpace} into centered planks of scale $\sigma$, writing $\theta(\xi) \in \Theta(\sigma,\eta)$ for $|\xi - \eta| \leq 4 r^{-1} \sigma^{-1}$. For $\omega \in \Omega_{\sigma}$ we have
\begin{equation*}
\big| \sum_{\theta \in \Theta_{r^{-2}}} \widehat{|f_{\theta}|^2}(\omega) \big|^2 \leq \big| \sum_{\Theta(\sigma,\eta) \in \mathbf{CP}_{\sigma}} \big| \sum_{\theta \in \Theta(\sigma,\eta)} \widehat{|f_{\theta}|^2}(\omega) \big| \big|^2.
\end{equation*}
Invoking Corollary \ref{cor:SortingPlanks} only the $\Theta(\sigma,\eta)$ are contributing, for which $\omega \in 10 \Theta(\sigma,\eta)$:
\begin{equation*}
\big| \sum_{\Theta(\sigma,\eta) \in \mathbf{CP}_{\sigma}} \big| \sum_{\theta \in \Theta(\sigma,\eta)} \widehat{|f_{\theta}|^2}(\omega) \big| \big|^2 \leq \big| \sum_{\substack{\Theta(\sigma,\eta) \in \mathbf{CP}_{\sigma}, \\ \omega \in 10 \Theta(\sigma,\eta)}} \big| \sum_{\theta \in \Theta(\sigma,\eta)} \widehat{|f_{\theta}|^2}(\omega) \big| \big|^2.
\end{equation*}
At this point, by Lemma \ref{lem:FiniteOverlapCP} we can apply Cauchy-Schwarz without essential loss:
\begin{equation*}
\int_{\Omega_{\sigma}} \big| \sum_{\substack{\Theta(\sigma,\xi) \in \mathbf{CP}_{\sigma}, \\ \omega \in 10 \Theta(\sigma,\xi)}} \big| \sum_{\theta \in \Theta(\sigma,\xi)} \widehat{|f_{\theta}|^2}(\omega) \big| \big|^2 d\omega \lesssim \int_{\Omega_{\sigma}} \sum_{\substack{\Theta(\sigma,\xi) \in \mathbf{CP}_{\sigma}, \\ \omega \in 10 \Theta(\sigma,\xi)}} \big| \sum_{\theta \in \Theta(\sigma,\xi)} \widehat{|f_{\theta}|^2}(\omega) \big|^2 d\omega.
\end{equation*}
This is further estimated by Plancherel's theorem:
\begin{equation*}
\int_{\Omega_{\sigma}} \sum_{\substack{\Theta(\sigma,\xi) \in \mathbf{CP}_{\sigma}, \\ \omega \in 10 \Theta(\sigma,\xi)}} \big| \sum_{\theta \in \Theta(\sigma,\xi)} \widehat{|f_{\theta}|^2}(\omega) \big|^2 d\omega \lesssim \sum_{\Theta(\sigma,\xi) \in \mathbf{CP}_{\sigma}} \int_{\R^{n+1}} \big| \sum_{\theta \in \Theta(\sigma,\xi)} |f_{\theta}|^2 (x) |^2 dx.
\end{equation*}
Now we can perceive $\theta \in \Theta(\sigma,\xi)$ as $\theta \subseteq 10 \tau$ for $\tau \in \Theta_{(r \sigma)^{-2}}$:
\begin{equation*}
\sum_{\Theta(\sigma,\xi) \in \mathbf{CP}_{\sigma}} \int_{\R^{n+1}} \big| \sum_{\theta \in \Theta(\sigma,\xi)} |f_{\theta}|^2 (x) |^2 dx \lesssim \sum_{\tau \in \Theta_{(\sigma r)^{-2}}} \int_{\R^{n+1}} \big| \sum_{\theta \subseteq 10 \tau} |f_{\theta}|^2(x) \big|^2 dx.
\end{equation*}
Plugging this into \eqref{eq:DecompositionSFFourierSpace} and redenoting $s = (r \sigma)^{-1}$ we find
\begin{equation*}
\eqref{eq:DecompositionSFFourierSpace} \lesssim \sum_{r^{-1} \leq s \leq 1} \sum_{\tau \in \Theta_{s^2}} \int_{\R^{n+1}} \big| \sum_{\theta \subseteq 10 \tau} |f_{\theta}|^2 (x) \big|^2  dx.
\end{equation*}
As the final step in the proof, we invoke the uncertainty principle. Note that the Fourier support of $\sum_{\theta \subseteq 10 \tau} |f_{\theta}|^2$ is contained in $U_{\tau,r^{2}}$ defined in \eqref{eq:PolarSetKakeyaSorting}.

We tile $\R^{n+1}$ with translates $U$ of $U_{\tau,r^2}$, denoted by $U \parallel U_{\tau,r^2}$, and take advantage of the fact that $|S_U f|$ is essentially constant on $U$. We conclude
\begin{equation*}
\begin{split}
\int_{\R^{n+1}} \big| \sum_{\theta \in \Theta_{r^{-2}}} |f_{\theta}|^2 \big|^2 &\lesssim \sum_{r^{-1} \leq s \leq 1} \sum_{\tau \in \Theta_{s^2}} \sum_{U \parallel U_{\tau,r^2}} \int_U \big| \sum_{\theta \subseteq 10 \tau} |f_{\theta}|^2(x) \big|^2 dx \\
&\lesssim \sum_{r^{-1} \leq s \leq 1} \sum_{\tau \in \Theta_{s^2}} \sum_{U \parallel U_{\tau,r^2}} |U|^{-1} \| S_U f \|^4_{L^2(U)}.
\end{split}
\end{equation*}
\end{proof}

\subsection{Induction on scales}

With the Kakeya, or wave envelope estimate, at hand, the proof of the square function estimate is carried out making use of the following two-scale-quantity:
\begin{definition}
Let $f \in \mathcal{S}(\R^{n+1})$ with $\text{supp}(\hat{f}) \subseteq \mathcal{N}_{R^{-1}}(\mathcal{C} M_Q)$. We define $S(r,R)$ as smallest constant such that
\begin{equation*}
\sum_{B_r \subseteq \R^{n+1}} |B_r|^{-1} \| S_{B_r} f \|^4_{L^2(B_r)} \leq S(r,R) \sum_{R^{-1} \leq s \leq 1} \sum_{\tau \in \Theta_{s}} \sum_{U \parallel U_{\tau,R}} |U|^{-1} \| S_U f \|^4_{L^2(U)}.
\end{equation*}
\end{definition}

Theorem \ref{thm:SquareFunctionConical} is a consequence of the following:
\begin{proposition}
\label{prop:WaveEnvelopeEstimate}
It holds
\begin{equation*}
S(1,R) \lesssim_\varepsilon R^\varepsilon.
\end{equation*}
\end{proposition}
Before we turn to the proof of the wave envelope estimate, we argue how it implies the square function estimate:
\begin{proof}[Proof~of~(Proposition~\ref{prop:WaveEnvelopeEstimate}~$\Rightarrow$~Theorem~\ref{thm:SquareFunctionConical})]
By the essentially constant property it holds for $f \in \mathcal{S}(\R^{n+1})$ with $\text{supp}(\hat{f}) \subseteq \mathcal{N}_{R^{-1}}(\mathcal{C} M_Q) \subseteq B_{n+1}(0,1)$ that
\begin{equation*}
\| f \|_{L^4(\R^{n+1})}^4 \lesssim \sum_{B_1 \in \mathcal{B}_1} |B_1|^{-1} \| S_{B_1} f \|^4_{L^2(B_1)}.
\end{equation*}

Applying the hypothesis yields
\begin{equation*}
\| f \|_{L^4(\R^{n+1})} \lesssim_\varepsilon R^\varepsilon \sum_{R^{-1} \leq s \leq 1} \sum_{\tau \in \Theta_{s}} \sum_{U \parallel U_{\tau,R}} |U|^{-1} \| S_U f \|^4_{L^2(U)}.
\end{equation*}
An application of the Cauchy-Schwarz inequality yields
\begin{equation*}
\sum_{U \parallel U_{\tau,R}} |U|^{-1} \big( \int_U \sum_{\theta \subseteq 10 \tau} |f_{\theta}|^2 \big)^2 \lesssim \int_{\R^{n+1}} \big( \sum_{\theta \subseteq 10 \tau} |f_{\theta}|^2 \big)^2.
\end{equation*}

Next, applying the embedding $\ell^1 \hookrightarrow \ell^2$ gives
\begin{equation*}
\int_{\R^{n+1}} \sum_{\tau \in \Theta_{s}} \big( \sum_{\theta \subseteq 10 \tau} |f_{\theta}|^2 \big)^2 \lesssim \int_{\R^{n+1}} \big( \sum_{\theta \in \Theta_R} |f_{\theta}|^2 \big)^2.
\end{equation*}
Carrying out the dyadic sum over $s$ incurs a logarithmic loss in $R$, which yields, subsuming the above estimates,
\begin{equation*}
R^{\varepsilon} \sum_{R^{-1} \leq s \leq 1} \sum_{\tau \in \Theta_{s}} \sum_{U \parallel U_{\tau,R}} |U|^{-1} \| S_U f \|^4_{L^2(U)} \lesssim R^\varepsilon \log(R) \int_{\R^{n+1}} \big( \sum_{\theta \in \Theta_R} |f_{\theta}|^2 \big)^2 dx.
\end{equation*}
The proof is complete.

\end{proof}

In the proof of Proposition \ref{prop:WaveEnvelopeEstimate} the following reformulation of Proposition \ref{prop:KakeyaEstimateConical}, making use of the locally constant property, plays an important role:
\begin{lemma}
\label{lem:LocalKakeyaEstimate}
Let $1 \leq r \leq r^2 \leq R$. It holds
\begin{equation*}
S(r,r^2) \lesssim 1.
\end{equation*}
\end{lemma} 

Moreover, we have the following inequality which relates two different scales via a generalized Lorentz transformation:
\begin{lemma}
\label{lem:TwoScaleInequality}
Let $1 \leq r_1 \leq r_2 \leq r_3$. Then the following estimate holds:
\begin{equation}
\label{eq:TwoScaleInequality}
S(r_1,r_3) \lesssim \log(r_2) S(r_1,r_2) \max_{s \in [r_2^{-1},1]} S(s r_2, s r_3).
\end{equation}
\end{lemma}

We digress for a moment to explain the key features of the generalized Lorentz transformation. Consider
\begin{equation*}
\mathcal{C} \Gamma = \{ (\xi, \frac{Q(\xi)}{h}, h) \in \R^d \times \R^l \times \R : \, \xi \in B_d(0,1), \, h \in [1/2,1] \}
\end{equation*}
with $Q(\xi) = \frac{1}{2} (\langle \xi, A_1 \xi \rangle, \ldots, \langle \xi, A_l \xi \rangle)$, $A_i \in \text{Sym}(\R^d)$. We denote subsets with base frequencies in an $s$-sector by
\begin{equation*}
\mathcal{C} \Gamma_{\tau,s} = \{ (\xi, Q(\xi)/h, h) : \xi \in B_d(0,1), \; \big| \frac{\xi}{h} - \xi_{\tau} \big| \leq s, \; h \in [1/2,1] \}.
\end{equation*}
Let $g \in \mathcal{S}(\R^{n+1})$ with $\text{supp}(\hat{g}) \subseteq \mathcal{N}_{r^{-2}}(\mathcal{C} \Gamma_{\tau,s})$, $s^{-1} \leq r$.
The generalized Lorentz transformation is supposed to enlarge the $s$-sector to the whole cone:
\begin{equation*}
\begin{split}
\Lambda_{\tau,s} : \R^d \times \R^l \times [1/2,1] \ni &(\xi,\zeta,h) \mapsto (\eta = s^{-1}(\xi - h \xi_{\tau}), \\
& (s^{-2} \zeta_i - s^{-1} \langle \eta, A_i \xi_{\tau} \rangle - \frac{1}{2} h s^{-2} \langle \xi_{\tau},A_i \xi_{\tau} \rangle)_{i=1,\ldots,l}, h).
\end{split}
\end{equation*}

We have $\Lambda_{\tau,s} (\mathcal{C} \Gamma_{\tau,s}) \approx \mathcal{C} \Gamma$ and moreover, $\Lambda_{\tau,s}(\mathcal{N}_{r^{-2}}(\mathcal{C} \Gamma_{\tau,s})) \approx \mathcal{N}_{r^{-2} s^{-2}}(\mathcal{C} \Gamma)$.

The classical case corresponds to $l=1$, $Q(\xi) = \frac{1}{2} \langle \xi, \xi \rangle$; see \cite[Section~5]{GuthWangZhang2020}.

\begin{proof}
In the first step we use the definition of $S(r_1,r_2)$:
\begin{equation*}
\sum_{B_{r_1} \in \mathcal{B}_{r_1}} |B_{r_1}|^{-1} \| S_{B_{r_1}} f \|^4_{L^2(B_{r_1})} \leq S(r_1,r_2) \sum_{r_2^{-\frac{1}{2}} \leq s \leq 1} \sum_{\tau \in \Theta_{s^2}} \sum_{U \parallel U_{\tau,r_2}} |U|^{-1} \| S_U f \|^4_{L^2(U)}.
\end{equation*}

We use the generalized Lorentz transformation $\Lambda_{\tau,s}$ centered at $\xi_{\tau}$ at scale $s$. 
The tiny subsectors $\big| \frac{\xi}{h} - \xi_{\theta} \big| \lesssim r_2^{-\frac{1}{2}}$, which constitute $S_U f = \big( \sum_{\theta \subseteq 10 \tau} |f_{\theta}|^2 \big)^{\frac{1}{2}}$, are mapped to the larger sectors
\begin{equation*}
\big| \frac{\xi}{h} - \xi^*_{\theta} \big| \lesssim r_2^{-\frac{1}{2}} s^{-1}.
\end{equation*}

Via the generalized Lorentz transformation $\Lambda_{\tau,s}$ we find that $\tau$ is mapped essentially to the full cone and the $r_3^{-1}$-neighborhood of $\mathcal{C} \Gamma_{\tau,s}$ is mapped to the $s^{-2} r_3^{-1}$-neighborhood of the full cone. Moreover, by definition it is easy to see that $\Lambda_{\tau,s}^*(U) \approx B_{s^2 r_2}$.

This implies that
\begin{equation*}
\sum_{U \parallel U_{\tau,r_2^2}} |U|^{-1} \| S_U f \|^4_{L^2(U)} \approx \sum_{B_{s^2 r_2} \in \mathcal{B}_{s^2 r_2}} |U|^{-1} |J \Lambda_{\tau}|^2 \| S_{B_{s^2 r_2}} \overline{f} \|^4_{L^2(B_{s^2 r_2})}.
\end{equation*}
Above $J \Lambda_{\tau}$ denotes the Jacobian incurred from the change of variables $x \mapsto \Lambda_{\tau,s}^* x$ and $\bar{f}$ denotes the function obtained from pulling back the Fourier transform of $f$ with $\Lambda_{\tau,s}$: $
\hat{\bar{f}} = \hat{f} \circ (\Lambda_{\tau,s})^{-1}$.

\smallskip

Now we can estimate the expression on the different scales $s^2 r_2$, $s^2 r_3$ to obtain
\begin{equation*}
\begin{split}
&\quad \sum_{B_{s^2 r_2} \in \mathcal{B}_{s^2 r_2}} |B_{s^2 r_2}|^{-1} \| S_{B_{s^2 r_2}} \bar{f} \|^4_{L^2(B_{s^2 r_2})} \\
&\leq S(s^2 r_2, s^2 r_3) \sum_{(s^2 r_3)^{-\frac{1}{2}} \leq \sigma \leq 1} \sum_{\tau \in \Theta_{\sigma^2}} \sum_{V \parallel V_{\tau,\sigma^2}} |V|^{-1} \| S_V \bar{f} \|^4_{L^2(V)}.
\end{split}
\end{equation*}
Undoing the change of variables by the properties of the generalized Lorentz transformation, we find
\begin{equation*}
\sum_{U \parallel U_{\tau,r_2^2}} |U|^{-1} \| S_U f \|^4_{L^2(U)} \lesssim S(s^2 r_2, s^2 r_3) \sum_{r_3^{-\frac{1}{2}} \leq \sigma \leq s} \sum_{\substack{\tau' \in \Theta_{\sigma^2}, \\ \tau' \subseteq \tau}} \sum_{V \parallel V_{\tau',\sigma^2}} |V|^{-1} \| S_V f \|^4_{L^2(V)}.
\end{equation*}

This allows us to write
\begin{equation*}
\begin{split}
&\quad \sum_{r_2^{-\frac{1}{2}} \leq s \leq 1} \sum_{\zeta \in \Theta_{s^2}} \sum_{U \parallel U_{\tau,s^2}} |U|^{-1} \| S_U f \|^4_{L^2(U)} \\
 &\lesssim \sum_{r_2^{-\frac{1}{2}} \leq s \leq 1} S(s^2 r_2, s^2 r_3) \sum_{\zeta \in \Theta_{s^2}} \sum_{r_3^{-\frac{1}{2}} \leq \sigma \leq s} \sum_{\substack{ \tau \in \Theta_{\sigma^2}, \\ \tau \subseteq \zeta}} \sum_{U \parallel U_{\tau,\sigma^2}} |U|^{-1} \| S_U f \|^4_{L^2(U)} \\
&\lesssim \log(r_2) \max_{s \in [r_2^{-1},1]} S(s r_2,s r_3) \sum_{r_{3}^{-\frac{1}{2}} \leq \sigma \leq 1} \sum_{\tau \in \Theta_{\sigma^2}} \sum_{U \parallel U_{\tau,\sigma^2}} |U|^{-1} \| S_U f \|^4_{L^2(U)}.
\end{split}
\end{equation*}
The logarithmic loss comes from counting $\tau \in \Theta_{\sigma^2}$ in the inner sum by the outer sum over $r_2$.
In conclusion, having redenoted $s^2 \to s$, we obtain
\begin{equation*}
S(r_1,r_3) \lesssim \log(r_2) S(r_1,r_2) \max_{r_2^{-1} \leq s \leq 1} S(s r_2, s r_3).
\end{equation*}
\end{proof}

The starting point of the induction on scales will be the square function estimates for quadratic manifolds. Choose $K \gg 1$ to be determined later and divide the height parameter into intervals of length $K^{-1}$. We write
\begin{equation*}
\mathcal{C}_K M_Q = \{ (\xi,M_Q(\xi)/h,h) \in \R^{n+1} : \xi \in  B_d(0,1), \quad h \in [a,a+K^{-1}] \}
\end{equation*}
for some $a \in [1/2,1]$ and define the modified quantity $S_K(r,R)$ analogous to the above with the difference that we are considering $f \in \mathcal{S}(\R^{n+1})$, $\text{supp}(\hat{f}) \subseteq \mathcal{N}_{\delta}(\mathcal{C}_K M_Q)$.

Note that we gloss over the dependence on $a$ as the following considerations are uniformly valid in $a$.

We have by simple decomposition of the Fourier support:
\begin{equation}
\label{eq:ReductionHeightLocalization}
S(r,R) \lesssim K S_K(r,R).
\end{equation}

The following lemma explains the introduction of the parameter $K$:
\begin{proposition}
\label{prop:InductionStart}
The following estimate holds for $1 \leq r \leq K$:
\begin{equation}
\label{eq:InductionStart}
S_K(1,K) \lesssim_\delta K^\delta.
\end{equation}
\end{proposition}
This is a variant of the square function estimate proved in Theorem \ref{thm:SquareFunctionQuadraticManifold}. Note that trivially Lemmas \ref{lem:TwoScaleInequality} and \ref{lem:LocalKakeyaEstimate} remain valid after replacing $S(r,R)$ with $S_K(r,R)$.

We explain the proof of Proposition \ref{prop:InductionStart}. The key ingredient is the following bilinear Kakeya estimate:
\begin{lemma}
Let $\theta_i(\xi_i) \in \Theta_{K^{-1}}$ be two conical sectors with separation of the base frequencies $|\xi_1-\xi_2| \sim s \gg K^{-\frac{1}{2}}$. Then we have for $\text{supp}(\widehat{f_{\theta_i}}) \subseteq \mathcal{N}_{\frac{1}{K}}(\mathcal{C} \Gamma_{\frac{1}{K}}) \cap \theta_i$:
\begin{equation}
\label{eq:BilinearKakeya}
\sum_{B_{K^{\frac{1}{2}}} \in \mathcal{B}_{K^{\frac{1}{2}}}} \int_{B_{K^{\frac{1}{2}}}} |f_{\theta_1} f_{\theta_2}|^2 \lesssim s^{-d} \sum_{B_K \in \mathcal{B}_K} |B_K|^{-1} \int |f_{\theta_1}|^2 w_{B_K} \int |f_{\theta_2}|^2 w_{B_K}.
\end{equation}
\end{lemma}
\begin{proof}
The argument rests on a wave packet decomposition:
\begin{equation*}
f_{\theta_i} = \sum_{T_i \parallel \tilde{\theta}_i^*} f_{\theta_i,T_i}
\end{equation*}
with $\tilde{\theta}_i^*$ being a rectangle of size
\begin{equation*}
\underbrace{K^{\frac{1}{2}} \times \ldots \times K^{\frac{1}{2}}}_{d \text{ components in } \mathbf{t}_j(\xi_i)} \times \underbrace{K \times \ldots \times K}_{l \text{ components in } \mathbf{n}_j(\xi_i)} \times K.
\end{equation*}
Denote $\sup_{x \in T_i} |f_{\theta_i,T_i}(x)| = c_{T_i}$ and recall that $|f_{\theta_i,T_i}(x)| \sim |f_{\theta_i,T_i}(y)|$ for $x,y \in T_i$.

We have to compute the size of the intersection of $T_1$ and $T_2$. To this end, we take advantage of the quantitative transversality of $\langle \mathbf{t}_1(\xi_1),\ldots,\mathbf{t}_d(\xi_1) \rangle$ and $\langle \mathbf{t}_1(\xi_2),\ldots,\mathbf{t}_d(\xi_2) \rangle$. We write
\begin{equation*}
\mathbf{t}_i(\xi_2) = \mathbf{t}_i(\xi_1) + \sum_{j=1}^d (\xi_2 - \xi_1)_j \partial_j \mathbf{t}_i(\xi_1)
\end{equation*}
to find by elimination that
\begin{equation*}
\begin{split}
&\quad \mathbf{t}_1(\xi_1) \wedge \ldots \wedge \mathbf{t}_d(\xi_1) \wedge \mathbf{t}_1(\xi_2) \wedge \ldots \wedge \mathbf{t}_d(\xi_2) \\
&= \mathbf{t}_1(\xi_1) \wedge \ldots \wedge \mathbf{t}_d(\xi_1) \wedge \begin{pmatrix}
0 \\ \vdots \\ 0 \\ (A_1 \Delta \xi)_1 \\ \vdots \\ (A_l \Delta \xi)_1
\end{pmatrix}
\wedge \ldots \wedge
\begin{pmatrix}
0 \\ \vdots \\ 0 \\ (A_1 \Delta \xi)_d \\ \vdots \\ (A_l \Delta \xi)_d
\end{pmatrix}
.
\end{split}
\end{equation*}
We arrange the vectors in an $n \times 2d$-matrix:
\begin{equation*}
\begin{pmatrix}
1 & 0 & \ldots & 0 & 0 & \ldots & 0 \\
0 & 1 & &\ldots & 0 & \ldots & 0 \\
\vdots &\vdots & \vdots & 0 & \vdots & 0 & \vdots \\
0 &0 & \ldots & 1 & 0 & \ldots & 0 \\
\partial_1 q_1(\xi_1) & \partial_2 q_1(\xi_1) & \ldots & \partial_d q_1(\xi_1) & \, & (A_1 \Delta \xi)^t & \, \\
 \vdots  & \vdots & \vdots & \vdots & \, & \vdots & \, \\
 \partial_1 q_l(\xi_1) & \partial_2 q_l(\xi_1) & \ldots & \partial_d q_l(\xi_1) & \, & (A_l \Delta \xi)^t & \,
\end{pmatrix}.
\end{equation*}
By the transversality assumption \eqref{eq:SpanCondition} the $2d$-dimensional volume is of size $\geq c |\Delta \xi|^d \gtrsim s^d$. This yields the estimate:
\begin{equation*}
|T_1 \cap T_2| \sim s^{-d} K^{\frac{d}{2}} K^{\frac{d}{2}} K^{l-d+1} \sim s^{-d} K^{l+1} \sim s^{-d} \frac{|T_1| |T_2|}{|B_K|}.
\end{equation*}

For this reason we can conclude
\begin{equation*}
\begin{split}
\sum_{B_{K^{\frac{1}{2}}} \subseteq B_K} \int_{B_{K^{\frac{1}{2}}}} |f_{\theta_1} f_{\theta_2}|^2 &\lesssim \sum_{\substack{B_{K^{\frac{1}{2}}} \subseteq B_K, \\ T_1 \parallel T_{\tilde{\theta_1^*}}, B_{K^{\frac{1}{2}}} \cap T_1 \neq \emptyset, \\ T_2 \parallel T_{\tilde{\theta_2^*}}, B_{K^{\frac{1}{2}}} \cap T_2 \neq \emptyset}} |B_{K^{\frac{1}{2}}}| c_{T_1}^2 c_{T_2}^2 \\
&\lesssim s^{-d} |B_K|^{-1} \big( \int_{B_K} \sum_{T_1 \parallel T_{\tilde{\theta_1^*}}} c_{T_1}^2 \chi_{T_1} \big) \big( \int_{B_K} \sum_{T_2 \parallel T_{\tilde{\theta_2^*}}} c_{T_2}^2 \chi_{T_2} \big) \\
&\lesssim  s^{-d} |B_K|^{-1} \int |f_{\theta_1}|^2 w_{B_K} \int |f_{\theta_2}|^2 w_{B_K}.
\end{split}
\end{equation*}
\end{proof}
The bilinear Kakeya estimate gives rise to the wave envelope estimate:
\begin{lemma}
Let $f \in \mathcal{S}(\R^{n+1})$ with $\text{supp}(\hat{f}) \subseteq \mathcal{N}_{\frac{1}{K}} ( \mathcal{C} \Gamma_{\frac{1}{K}})$. For any $\delta > 0$ it holds
\begin{equation}
\label{eq:WaveEnvelopeL4}
\| f \|^4_{L^4(\R^{n+1})} \leq C_\delta K^{\delta} \sum_{K^{-\frac{1}{2}} \leq s \leq 1} \sum_{\tau \in \Theta_{s^2}} \sum_{U \parallel U_{\tau,K}} |U|^{-1} \| S_U f \|^4_{L^2(U)}.
\end{equation}
\end{lemma}
\begin{proof}
This can be proved following along the lines of \cite[Lemma~6.5]{GuthWangZhang2020}, employing the Bourgain--Guth \cite{Bourgainguth2011} argument and the bilinear Kakeya estimate from the preceding lemma.
\end{proof}

Proposition \ref{prop:InductionStart} is a localized variant of the previous lemma. The proof of Proposition \ref{prop:InductionStart} additionally uses the biorthogonality to bound the lhs of \eqref{eq:InductionStart} by \eqref{eq:WaveEnvelopeL4}. We omit the details from \cite{GuthWangZhang2020} to avoid repetition.

\smallskip

We are ready to show Proposition \ref{prop:WaveEnvelopeEstimate} by induction-on-scales:
\begin{proof}[Proof~of~Proposition~\ref{prop:WaveEnvelopeEstimate}]
We shall show the stronger estimate
\begin{equation}
\label{eq:InductionEstimate}
S_K(r,R) \leq C_\varepsilon \big( \frac{R}{r} \big)^\varepsilon.
\end{equation}
which will suffice for $K= R^{\delta}$, $\delta \ll \varepsilon$. Moreover, note that we can choose $R \gg 1$ from the outset.

Suppose \eqref{eq:InductionEstimate} holds for $\frac{R}{r} \leq \frac{R_0}{2}$. We shall show that the estimate remains valid for $\frac{R}{r} \leq R_0$. The base case is $R_0 \leq K$, for which we have
\begin{equation*}
S_K(r,R) \leq C
\end{equation*}
as a consequence of Lemmas \ref{lem:TwoScaleInequality} and \ref{lem:LocalKakeyaEstimate}.

So, suppose that $R_0 > K$. In case $r \leq K^{\frac{1}{2}}$ we use Lemma \ref{lem:TwoScaleInequality} to find
\begin{equation*}
S_K(r,R) \lesssim S_K(r,K) \log(K) \max_{s \in [K^{-1},1]} S_K(s K , s R) \lesssim \frac{\log(K)}{K^\varepsilon} R^\varepsilon.
\end{equation*}
The ultimate estimate is due to the induction assumption. With $K= R^{\delta}$ and $R \gg 1$ induction readily closes.

Secondly, suppose that $r \geq K^{\frac{1}{2}}$. In this case we use Lemmas \ref{lem:LocalKakeyaEstimate} and \ref{lem:TwoScaleInequality} to find
\begin{equation*}
S_K(r,R) \lesssim S_K(r,r^2) \log(r) \max_{s \in [r^{-2},1]} S_K(s r^2, s R) \leq \log(r) C_\varepsilon \big( \frac{R}{r^2} \big)^{\varepsilon}
\end{equation*}
and induction closes by $\log(r) / r^\varepsilon \ll 1$ since $r \geq K^{\frac{1}{2}} \geq R^{\frac{\delta}{2}}$.
\end{proof}

\section{Local smoothing estimates for multi-parameter averages}
\label{section:MultiParameterAverages}

Next, we point out how the square function estimates yield local smoothing estimates for multi-parameter averages. Multi-parameter averages were studied by Marletta--Ricci \cite{MarlettaRicci1998} (see also \cite{MarlettaRicciZienkiewicz1998}). Presently, we are dealing with multi-parameter averages over quadratic hypersurfaces satisfying a transversality condition.

\medskip

We have the following set-up:
Let $f \in \mathcal{S}(\R^{d+1})$ and $\gamma_1,\ldots,\gamma_m:\R^d \to \R^{d+1}$ be graph parametrizations of quadratic functions:
\begin{equation*}
\gamma_i(s_1,\ldots,s_d) = (s_1,\ldots,s_d,q_i(s_1,\ldots,s_d)), \quad q_i(s_1,\ldots,s_d) = \frac{\langle s, A_i s \rangle}{2}, \quad A_i \in \R^{d \times d}.
\end{equation*}
We assume the non-degeneracy and symmetry conditions:
\begin{equation}
\label{eq:NonDegeneracy}
\det(A_i) \neq 0, \quad A_i = A_i^t.
\end{equation}
In addition we suppose the transversality assumption for the inverted quadratic forms: 

\smallskip

\emph{There is $c>0$ such that for any $\nu \in \mathbb{S}^d$ there are $1 \leq i_1 < i_2 < \ldots < i_d \leq m$:}
\begin{equation}
\label{eq:TransversalityAssumptionAverages}
|\det (A_{i_1}^{-1} \nu, \ldots , A_{i_d}^{-1} \nu)| \geq c > 0.
\end{equation}

\smallskip

The multi-parameter averages we consider are given by
\begin{equation}
\label{eq:MultiParameterAverages}
\mathcal{A}_{t} f(x) = \int f(x-\sum_{i=1}^m t_i \gamma_i(s_1^{(i)},\ldots,s_d^{(i)}) ) \chi(s^{(1)}) \ldots \chi(s^{(m)}) ds^{(1)} \ldots ds^{(m)}
\end{equation}
with $\chi \in C^\infty_c(B_d(0,1))$.

In the simplest non-trivial case, $d=1$, $\gamma(s) = (s,s^2)$ we have for $f \in \mathcal{S}(\R^2)$:
\begin{equation}
\label{eq:OneDimLocalSmoothing}
\mathcal{A}^{(1)}_t f(x) = \int f(x-t(s,s^2)) \chi(s) ds,
\end{equation}
which is closely related to the propagation of the wave equation in $2+1$-dimensions. Indeed, the wave propagation involves averages over spheres. For comparison recall the sharp local smoothing estimate for the wave equation in $2+1$-dimensions proved by Guth--Wang--Zhang \cite{GuthWangZhang2020}, which reads at the endpoint:
\begin{equation*}
\| \chi_1(t) \mathcal{A}^{(1)}_t f \|_{L_{t,x}^4(\R \times \R^2))} \lesssim_\varepsilon \| f \|_{L^4_\varepsilon(\R^2)}.
\end{equation*}
For technical convenience we consider $\chi_1 \in C^\infty_c(1/2,3/2)$.

\medskip

Presently, we consider the following extension of \eqref{eq:OneDimLocalSmoothing} given in \eqref{eq:TwoParamLocalSmoothing}:
\begin{equation*}
\mathcal{A}^{(2)}_t f(x) = \int f(x- t_1 \gamma_1(s^{(1)}) - t_2 \gamma_2(s^{(2)})) \chi(s^{(1)}) \chi(s^{(2)}) ds^{(1)} ds^{(2)}
\end{equation*}
with
\begin{equation*}
\gamma_1(s) = (s_1,s_2,s_1^2 - s_2^2), \quad \gamma_2(s) = (s_1,s_2,2s_1 s_2).
\end{equation*}

The local smoothing estimate claimed in \eqref{eq:MultiParameterLocalSmoothing} reads:
\begin{equation*}
\| \chi_1(t_1) \chi_1(t_2) \mathcal{A}^{(2)}_{t} f \|_{L^4_{t,x}(\R^{5})} \lesssim_\varepsilon \| f \|_{L^4_{-2+\varepsilon}(\R^3)}.
\end{equation*}


The proof proceeds in three steps:
\begin{itemize}
\item we obtain an oscillatory integral representation, which generalizes the FIO representation for \eqref{eq:OneDimLocalSmoothing},
\item the multi-parameter FIO representation satisfies the conditions of Theorem \ref{thm:SquareFunctionConical} such that we can apply a square function estimate,
\item after invoking duality, the proof is concluded using maximal function estimates, which extend the scalar case.
\end{itemize}
Most of the arguments are valid under the general assumptions given in \eqref{eq:NonDegeneracy} and \eqref{eq:TransversalityAssumptionAverages}. It is merely for our proof of the maximal function estimate in the last step we use a more specific property of the two-parameter average in \eqref{eq:TwoParamLocalSmoothing}. Possibly, a different proof of the maximal function estimate in Section \ref{subsection:MaximalFunction} could rely only on \eqref{eq:NonDegeneracy} and \eqref{eq:TransversalityAssumptionAverages}, which naturally appear several times in the argument. Most of the arguments below are carried out for the outset described by \eqref{eq:MultiParameterAverages} under assumptions \eqref{eq:NonDegeneracy} and \eqref{eq:TransversalityAssumptionAverages}.

\subsection{Oscillatory integral representation}
 
We use Fourier inversion and Fubini's theorem to write
\begin{equation*}
\begin{split}
&\quad \int f(x-\sum_{i=1}^m t_i \gamma_i(s_1^{(i)},\ldots,s_d^{(i)}) ) \chi(s^{(1)}) \ldots \chi(s^{(m)}) ds^{(1)} \ldots ds^{(m)} \\
&= \int e^{i \langle x, \xi \rangle} \hat{f}(\xi) \big( \prod_{i=1}^m \int e^{-it_i \gamma_i(s^{(i)}) \cdot \xi} \chi(s^{(i)}) ds^{(i)} \big) d\xi.
\end{split}
\end{equation*}
By Littlewood-Paley decomposition, it suffices to show a frequency-localized estimate for $N \gg 1$:
\begin{equation*}
\| \mathcal{A}_t (P_N f) \|_{L^4_{t,x}(\R^{d+1})} \lesssim_\varepsilon N^{\varepsilon+s} \| f \|_{L^4_x},
\end{equation*}
where $s=-2$ under the assumptions of Theorem \ref{thm:MultiParameterSmoothing}. Indeed, the small frequencies are estimated by simple kernel estimates.

\smallskip

We compute the asymptotics of the $s$-integrals by the following proposition, in which we recall the asymptotics of parameter-dependent stationary phase as described in \cite[Corollary~1.1.8]{Sogge2017}:

\medskip

Suppose that $\Psi: \R^d \times \R^{d+1} \to \R$, $(s,\xi) \mapsto \Psi(s,\xi)$ is a smooth function, which satisfies for $\xi \in U \subseteq \R^d$:
\begin{equation*}
\nabla_s \Psi(\xi,s_{\xi}) = 0
\end{equation*}
and
\begin{equation*}
\det \partial^2_s \Psi(\xi,s_{\xi}) \neq 0.
\end{equation*}
Then, by the implicit function theorem there is a smooth function $s(\xi)$ such that
\begin{equation*}
\nabla_s \Psi(\xi,s(\xi)) = 0.
\end{equation*}
Consider the oscillatory integral
\begin{equation*}
I(\xi,t) = \int_{\R^d} e^{i t \Psi(s,\xi)} a(t,s,\xi) ds
\end{equation*}
with $a$ having small enough $\xi$-support such that $s(\xi)$ is the only stationary point for $\xi$ and $a$ satisfies the symbol bounds:
\begin{equation*}
| \partial_t^m \partial_{\xi}^{\alpha} \partial_{s}^{\beta} a(t,\xi,s)| \lesssim (1+|t|)^{-|\alpha|}.
\end{equation*}

We have the following:
\begin{proposition}
\label{prop:OscillatoryIntegralExpansion}
Suppose that $a$ and $\Psi$ are as above. Then for every $\alpha \in \N_0$ and $\beta \in \N_0^d$ we have
\begin{equation*}
\big| \partial_t^{\alpha} \partial_{\xi}^{\beta} \big( e^{-i t \Psi(\xi,s)} I(\xi,t) \big) \big| \leq C_{\alpha \gamma} (1+|t|)^{-\frac{n}{2}-\alpha}.
\end{equation*}
\end{proposition}


In the following let $\xi = (\xi',\xi_{d+1}) \in \R^{d} \times \R$. We apply the above proposition to the phase function:
\begin{equation*}
\Phi_i(s_1,\ldots,s_d,\xi)  = \gamma_i(s_1,\ldots,s_d) \cdot \xi = \sum_{i=1}^d s_i \xi_i + q_i(s) \xi_{d+1}.
\end{equation*}
Note that for $|\xi_{d+1}| \ll |\xi'|$ there is no stationary point and we have rapid decay:
\begin{equation*}
\big| \int e^{-i t_i \Phi_i(s,\xi)} \chi(s) ds \big| \lesssim (1+|\xi|)^{-N}.
\end{equation*}
Hence, the essential contribution stems from $|\xi_{d+1}| \sim |\xi| \sim N \gg 1$. We have the following:
\begin{lemma}
For $|\xi_{d+1}| \sim |\xi| \sim N \gg 1$ and $t \sim 1$ we have the following:
\begin{equation*}
\int e^{-it \gamma_i(s^{(i)}) \cdot \xi} \chi(s^{(i)}) ds^{(i)} =  e^{i t \langle A_i^{-1} \xi', \xi' \rangle/4 \xi_{d+1}} a_i(t,\xi)
\end{equation*}
with
\begin{equation*}
| \partial_t^{\alpha} \partial_{\xi}^{\beta} a_i(t,\xi)| \lesssim_{\alpha,\beta} (1+|\xi|)^{-\frac{d}{2} - |\beta|}
\end{equation*}
\end{lemma}
\begin{proof}
We compute the stationary point:
\begin{equation*}
\nabla_s \Phi_i(s^*,\xi) = 0 \Leftrightarrow \xi' + 2 A_i(s^*) \xi_{d+1} = 0.
\end{equation*}
Hence, by the non-degeneracy of $A_i$, we find
\begin{equation*}
s^* = - A_i^{-1} \xi'/ 2 \xi_{d+1}.
\end{equation*}
We obtain
\begin{equation*}
\Phi_i(s^*,\xi) = - \langle A_i^{-1} \xi', \xi' \rangle/2 \xi_{d+1} + \langle A_i^{-1} \xi', \xi' \rangle/4 \xi_{d+1} = - \langle A_i^{-1} \xi', \xi' \rangle/4 \xi_{d+1}.
\end{equation*}
We obtain the expansion (see Proposition \ref{prop:OscillatoryIntegralExpansion})
\begin{equation*}
\int e^{-i t_i \gamma_i(s^{(i)}) \cdot \xi} \chi(s^{(i)}) ds^{(i)} = e^{i t_i \langle A_i^{-1} \xi', \xi' \rangle/4 \xi_{d+1}} a_i(t,\xi) \text{ for } |\xi_{d+1}| \gtrsim |\xi'| \gtrsim 1
\end{equation*}
with
\begin{equation*}
| \partial_t^{\alpha} \partial_{\xi}^{\beta} a_i(t,\xi)| \lesssim_{\alpha,\beta} (1+|\xi|)^{-\frac{d}{2}}.
\end{equation*}
\end{proof}

Let $\beta \in C^{\infty}_c((B_{d+1}(0,2) \backslash B_{d+1}(0,1/2))$. Plugging in the expansion of the integrals in $s$ we have obtained a multi-parameter Fourier integral operator after microlocalization:
\begin{equation*}
\begin{split}
\mathcal{A}_t P_N f(x) &= \int e^{i \langle x, \xi \rangle} \beta_N(\xi) \hat{f}(\xi) \big( \prod_{i=1}^m \int e^{-it_i \gamma_i(s^{(i)}) \cdot \xi} \chi(s^{(i)}) ds^{(i)} \big) d\xi \\
&= \int e^{i ( \langle x, \xi \rangle + \sum_{i=1}^m t_i \langle A_i^{-1} \xi',\xi' \rangle/4 \xi_{d+1} )} a(t,\xi) \beta_N(\xi) \hat{f}(\xi) d\xi.
\end{split}
\end{equation*}
Now we expand $a$ in a Fourier series in $t$ 
\begin{equation*}
a(t,\xi) = \tilde{\chi}(t) \sum_{k \in \Z^{m}} e^{i \langle k , t \rangle} \hat{a}(k,\xi)
\end{equation*}
for a suitable cutoff $\tilde{\chi} \in C^\infty_c$ and absorb $a$ into $\beta$ to dominate
\begin{equation*}
| \mathcal{A}_t P_N f|(x) \lesssim N^{- \frac{m d}{2}} \big| \int e^{i ( \langle x, \xi \rangle + \sum_{i=1}^m t_i \langle A_i^{-1} \xi',\xi' \rangle/4 \xi_{d+1} )} \tilde{\beta}_N(\xi) \hat{f}(\xi) d\xi \big|
\end{equation*}
for some $\tilde{\beta} \in C^\infty_c(B_{d+1}(0,2) \backslash B_{d+1}(0,1/2))$.

\smallskip

Let $\Phi:\R^{d+1} \times \R^m \times \R^{d+1} \to \R$ with
\begin{equation*}
\Phi(x,t;\xi) = \langle x, \xi \rangle + \sum_{i=1}^m t_i \langle A_i^{-1} \xi',\xi' \rangle/4 \xi_{d+1}
\end{equation*}
and focus in the following on the estimate 
\begin{equation}
\label{eq:FourierIntegralReduction}
\| \chi(t) \mathcal{F} f \|_{L^4_{t,x}(\R^n)} \lesssim_\varepsilon N^{\varepsilon} \| f \|_{L^4_x} \text{ with }
\mathcal{F} f(t,x) = \int e^{i \Phi(x,t;\xi)} \beta_N(\xi) \hat{f}(\xi) d\xi
\end{equation}
and $\chi(t) = \prod_{i=1}^m \chi_1(t_i)$, $\chi_1 \in C^\infty_c(1/2,3/2)$. Note that for Theorem \ref{thm:MultiParameterSmoothing} the smoothing of order $-2 = md/2$ stems from the symbol.

 We summarize our findings in the following proposition:
\begin{proposition}
The multi-parameter local smoothing estimate \eqref{eq:MultiParameterLocalSmoothing} is implied by \eqref{eq:FourierIntegralReduction}.
\end{proposition}

Consequently, with $|\xi'| \lesssim |\xi_{d+1}| \sim N \gg 1$ the Fourier transform of $\chi(t) \mathcal{F} f$ in all variables is in the unit neighborhood of the graph $\Gamma \mathcal{M}_{q,N}$:
\begin{equation*}
\xi \mapsto (\xi_1,\ldots,\xi_{d+1},\frac{\langle A_1^{-1} \xi', \xi' \rangle}{4 \xi_{d+1}},\ldots,\frac{\langle A_m^{-1} \xi', \xi' \rangle}{4 \xi_{d+1}} ); \quad \xi \in \R^{d+1}, \, |\xi_{d+1}| \sim |\xi| \sim N.
\end{equation*}
This is a multi-parameter cone over the base curve:
\begin{equation*}
\xi' \mapsto (\xi', \frac{\langle A_1^{-1} \xi', \xi' \rangle}{4}, \ldots, \frac{\langle A_m^{-1} \xi', \xi' \rangle}{4}); \quad \xi' \in \R^d, \; |\xi'| \lesssim N.
\end{equation*}
The transversality assumption \eqref{eq:TransversalityAssumptionAverages} will allow us to invoke the square function estimates for cones over quadratic manifolds obtained in previous sections.

\subsection{Applying the square function estimate}

By non-stationary phase and the compact parameter range $|t| \leq 1$, we can localize the support of $f$ to unit size. More precisely, by Fourier inversion let $\mathcal{F} f = \int K_N(t,x-y) f(y) dy$. We obtain the following simple kernel estimate:
\begin{lemma}
For $|x| \gg 1$ it holds the following estimate:
\begin{equation*}
\chi(t) |K_N(t,x)| \lesssim_N (1+|x|)^{-N}.
\end{equation*}
\end{lemma}

Then, after rescaling $\xi \to \xi/N$, $x \to N x$, $t \to N x$ and invoking translation invariance, it suffices to show the smoothing estimate:
\begin{equation*}
\big\| \int e^{i \Phi(x,t;\xi)} \chi(\xi) \hat{f}(\xi) d\xi \big\|_{L^4_{t,x}(B_{n}(0,N))} \lesssim_\varepsilon N^{\frac{m}{4}+\varepsilon} \| f \|_{L^4_x}.
\end{equation*}
Above $n =d+1+m$, $\chi \in C^\infty_c(B_{d+1}(0,2))$ with $|\xi_{d+1}| \sim |\xi'| \sim 1$ for $(\xi',\xi_{d+1}) \in \text{supp}(\chi)$. We dominate
\small
\begin{equation*}
\big\| \int e^{i \Phi(x,t;\xi)} \chi(\xi) \hat{f}(\xi) d\xi \big\|_{L^4_{t,x}(B_{d+1+m}(0,N))} \lesssim \big\| \chi_N(t,x) \int e^{i \Phi(x,t;\xi)} \chi(\xi) \hat{f}(\xi) d\xi \big\|_{L^4_{t,x}(\R^n)}
\end{equation*}
\normalsize
with $\text{supp}(\widehat{\chi_N}) \subseteq B(0,c/N)$. The function $F_N(t,x) = \chi_N(t,x) \int e^{i \Phi(x,t;\xi)} \chi(\xi) \hat{f}(\xi) d\xi$ satisfies $\text{supp}(\widehat{F_N}) \subseteq \mathcal{N}_{\frac{c}{N}}(\Gamma \mathcal{M}_{q,1})$.

Consequently, we may apply Theorem \ref{thm:SquareFunctionConical} to obtain
\begin{equation*}
\| F_N \|_{L^4_{t,x}(\R^n)} \lesssim_\varepsilon N^\varepsilon \big\| \big( \sum_{\theta \in \Theta_{\frac{1}{N}}} |F_{N \theta}|^2 \big)^{\frac{1}{2}} \big\|_{L^4_{t,x}(\R^n)}
\end{equation*}
with $\Theta_{\frac{1}{N}} = \Theta_{\frac{1}{N}}(\Gamma \mathcal{M}_{q,1})$ denoting the canonical covering of the $1/N$-neighborhood of $\Gamma \mathcal{M}_{q,1}$ introduced in the previous sections.

We reverse the scaling and by standard arguments, we reduce to establishing the estimate:
\begin{equation}
\label{eq:MultiParameterSquareFunctionReduction}
\big\| \chi_1(t) \big( \sum_{\theta \in \Theta_{\frac{1}{N}}} \big| \int e^{i \Phi(x,t;\xi)} \chi_{\theta}(\xi) \beta_N(\xi) \hat{f}(\xi) d\xi \big|^2 \big)^{\frac{1}{2}} \big\|_{L^4_{t,x}(\R^n)} \lesssim_\varepsilon N^{\varepsilon} \| f \|_{L^4}
\end{equation}
with
\begin{equation*}
\big| \frac{\xi'}{\xi_{d+1}} - \xi'_\theta \big| \lesssim N^{-\frac{1}{2}} \text{ for } \xi \in \text{supp}(\chi_{\theta}).
\end{equation*}

\subsection{Conclusion via maximal function estimates}
\label{subsection:MaximalFunction}
The frequency projection to $\theta$ reduces the propagation with $\mathcal{F}$ to a translation. By a duality argument, we reduce \eqref{eq:MultiParameterSquareFunctionReduction} to a maximal function estimate involving averages over the translations induced by $\mathcal{F}$ on $\theta$-frequencies.
The proof of \eqref{eq:MultiParameterSquareFunctionReduction} becomes a variant of the arguments in \cite{MockenhauptSeegerSogge1992}.

\smallskip

\subsubsection{Reduction to a maximal function estimate}
Squaring the left hand-side we find by duality
\begin{equation}
\label{eq:MaximalFunctionReductionI}
\begin{split}
&\quad \int \big( \sum_{\theta \in \Theta_N} \big| \int_{\theta: | \frac{\xi'}{\xi_{d+1}} - \xi'_{\theta} |  \lesssim N^{-\frac{1}{2}}} e^{i \Phi(x,t;\xi)} \chi(\xi) \hat{f}(\xi) d\xi \big|^2 \big)^2 dx dt
\\
 &= \sup_{\| g \|_{L^2} = 1} \sum_{\theta} \big| \iint K_{\theta}(x,t;y) f_\theta(y) dy \big|^2 \; \bar{g}(x,t) dx dt
 \end{split}
 \end{equation}
And further
\begin{equation}
\label{eq:MaximalFunctionReductionII}
\begin{split} 
\eqref{eq:MaximalFunctionReductionI} &\lesssim \sup_{\| g \|_2 = 1} \int_{\R^{d+1}} \sum_{\theta} |f_{\theta}|^2(y) \big\{ \iint |K_{\theta}(x,t;y)| |g(x,t)| dx dt \big\} dy \\
&\lesssim \big( \int_{\R^{d+1}} \big( \sum_{\theta} |f_{\theta}|^2 \big)^2 \big)^{\frac{1}{2}} \big( \int_{\R^{d+1}} |\sup_{\theta} \iint |K_{\theta}(x,t;y)| |g(x,t)| dx dt |^2 dy \big)^{\frac{1}{2}} \\
&\lesssim \| f \|^2_{L^4} \| \sup_{\theta} |A_{\theta} g|(y) \|_{L^2(\R^{d+1})}.
\end{split}
\end{equation}
In the ultimate estimate we have set for $\theta \in \Theta_N$:
\begin{equation*}
A_{\theta} g(y) = \iint K_{\theta}(x,t;y) g(x,t) dx dt 
\end{equation*}
and used a standard square function estimate in $L^4$. This square function estimate is a variant of Littlewood-Paley orthogonality.

\smallskip

We lay the last brick in the proof of Theorem \ref{thm:MultiParameterSmoothing}, which is the following lemma:
\begin{lemma}
\label{lem:MaximalFunctionEstimate}
It holds
\begin{equation*}
\| \sup_{\theta \in \Theta_N} |A_{\theta} g|(y) \|_{L^2(\R^3)} \lesssim_\varepsilon N^{\varepsilon} \| g \|_{L^2_{x,t}(\R^5)}.
\end{equation*}
\end{lemma}
Indeed, plugging the key identity into \eqref{eq:MaximalFunctionReductionII} we obtain
\begin{equation*}
\eqref{eq:MaximalFunctionReductionII} \lesssim_\varepsilon N^{\varepsilon} \| f \|^2_{L^4}.
\end{equation*}
This will complete the proof of Theorem \ref{thm:MultiParameterSmoothing}.

To carry out the proof of Lemma \ref{lem:MaximalFunctionEstimate}, we take a look at the kernel:
\begin{equation*}
K_{\theta}(t,x;y) = \int e^{i( \langle x-y, \xi \rangle+ t_1 \frac{\langle \xi', A_1^{-1} \xi' \rangle}{\xi_{d+1}} + \ldots + t_m \frac{\langle \xi', A_m^{-1} \xi' \rangle}{\xi_{d+1}} )} \chi_{\theta}(\xi) \beta(|\xi|/N) d\xi
\end{equation*}
with $\big| \frac{\xi'}{\xi_{d+1}} - \xi'_{\theta}| \lesssim N^{-\frac{1}{2}}$, $|\xi_{d+1}| \sim N$ within the support of $\chi_{\theta}$. We obtain for the phase function for $\xi \in \text{supp}(\chi_{\theta}) \cap \text{supp}(\beta_N)$:
\begin{equation*}
\begin{split}
\Phi(x-y,\xi) &= \langle x-y, \xi \rangle + t_1 (2 A_1^{-1} \xi'_{\theta}, - \langle A_1^{-1} \xi'_{\theta}, \xi'_{\theta} \rangle) \cdot \xi + \ldots \\
&\quad + t_m (2 A_m^{-1} \xi'_{\theta}, - \langle A_m^{-1} \xi'_{\theta}, \xi'_{\theta} \rangle) \cdot \xi + \mathcal{O}(|t| N^{-1} |\xi|).
\end{split}
\end{equation*}
This means that the averaging over $|K_{\theta}|$ as in $A_\theta$ describes an averaging over $2 A_i^{-1} \xi'_{\theta}$ (parametrized by $t_i$) in directions $x'$ and $- \langle A_1^{-1} \xi'_{\theta}, \xi'_{\theta} \rangle - \ldots - \langle A_m^{-1} \xi'_{\theta}, \xi'_{\theta} \rangle$ in direction $x_{d+1}$.

We define the localized operators:
\begin{equation*}
A_{\theta}^{N,k} g(y) = C \iint e^{i \bar{\Phi}(x-y,\xi)} \beta(|\xi|/N) \beta \big( 2^{-2k} N \big| \frac{\xi'}{\xi_{d+1}} - \xi'_{\theta} \big|^2 \big) \hat{g}(\xi,t) d\xi dt
\end{equation*}
with
\begin{equation*}
\begin{split}
\bar{\Phi}(x-y,\xi) &= \langle x- y, \xi \rangle + t_1 (2A_1^{-1} \xi'_{\theta}, - \langle A_1^{-1} \xi'_{\theta}, \xi'_{\theta} \rangle) \cdot \xi \\
&\quad + t_2 (2 A_2^{-1} \xi'_{\theta}, - \langle A_2^{-1} \xi'_{\theta}, \xi'_{\theta} \rangle) \cdot \xi  + \ldots + t_m (2A_m^{-1} \xi'_{\theta}, - \langle A_m^{-1} \xi'_{\theta}, \xi'_{\theta} \rangle) \cdot \xi.
\end{split}
\end{equation*}

The summation over $\{ k : 2^k \lesssim N^{\frac{1}{2}} \}$ will incur a logarithmic loss in $N$ provided that we obtain uniform bounds in $k$. To obtain an estimate for
\begin{equation*}
\sup_{\theta} |A_{\theta}^{N,k} g|^2 (y),
\end{equation*}
we use Sobolev embedding to reduce to an estimate for $\| A_{\theta}^{N,k} g \|_{L^2_{\theta} L^2_y}$. 

The following computations are restricted to the special case considered in Theorem \ref{thm:MultiParameterSmoothing}. For Sobolev embedding, we compute the $\theta$-frequencies for $A_{\theta}^{N,k} g$:
\small
\begin{equation*}
\begin{split}
&\quad \int A_{\theta}^{N,k} g(y) e^{-i \xi'_{\theta} \tau} d \xi'_{\theta} \\
&= \int e^{i( \langle x-y, \xi \rangle + t_1 (2 A_1^{-1} \xi'_{\theta}, - \langle \xi'_{\theta}, A_1^{-1} \xi'_{\theta} \rangle) \cdot \xi + t_2 ( 2 A_2^{-1} \xi'_{\theta}, - \langle \xi'_{\theta}, A_2^{-1} \xi'_{\theta} \rangle) \cdot \xi - i \xi'_{\theta} \tau)} \beta_{N,k}(\xi) \hat{g}(\xi) d\xi d \xi'_{\theta}.
\end{split}
\end{equation*}
\normalsize
We compute the stationary point of the phase in $\xi'_{\theta}$ as
\begin{equation*}
\psi(\xi'_{\theta}) = t_1 (2 A_1^{-1} \xi'_{\theta}, - \langle \xi'_{\theta}, A_1^{-1} \xi'_{\theta} \rangle) \cdot \xi + t_2 ( 2 A_2^{-1} \xi'_{\theta}, - \langle \xi'_{\theta}, A_2^{-1} \xi'_{\theta} \rangle) \cdot \xi) - \xi'_{\theta} \cdot \tau
\end{equation*}
and
\begin{equation*}
\begin{split}
\nabla_{\xi'_{\theta}} \psi(\xi'_{\theta}) &= t_1 (2 A_1^{-1} \xi' - 2 A_1^{-1} \xi'_{\theta} \xi_3) + t_2 (2 A_2^{-1} \xi' - 2 A_2^{-1} \xi'_{\theta} \xi_3) - \tau \\
&= t_1 \xi_3 ( 2 A_1^{-1} \big( \frac{\xi'}{\xi_3} - \xi'_{\theta} \big) ) + t_2 \xi_3 ( 2 A_2^{-1} \xi' - 2 A_2^{-1} \xi'_{\theta} \xi_3) - \tau.
\end{split}
\end{equation*}
Keeping in mind that $|t_i| \lesssim 1$, $|\xi_3| \sim N$, and $\big| \frac{\xi'}{\xi_3} - \xi'_{\theta} \big| \lesssim N^{-\frac{1}{2}} 2^k$, we have for $|\tau| \gg N^{\frac{1}{2}+\varepsilon} 2^k$ that $|\nabla_{\xi'_{\theta}} \psi(\xi'_{\theta})| \sim |\tau|$. This gives rapid decay for $|\tau| \gg N^{\frac{1}{2}+\varepsilon} 2^k$.

With essential boundedness of the frequencies by $\lesssim N^{\frac{1}{2}} 2^k$ an application of Sobolev embedding $W^{1,2+\varepsilon}(\R^2) \hookrightarrow L^{\infty}(\R^2)$ gives
\begin{equation}
\label{eq:SobolevEmbeddingMaximalFunction}
\| \sup_{\theta} |A_{\theta}^{N,k} g|^2(y) \|_{L^1_y} \lesssim_\varepsilon N^{\varepsilon} N 2^{2k} \| A_{\theta}^{N,k} g \|^2_{L^2_{\theta} L^2_{y}}.
\end{equation}

We turn to the key estimate for $\| A_{\theta}^{N,k} g(y) \|_{L^2_{y} L^2_{\theta}}$ by rewriting the square as:
\small
\begin{equation*}
\begin{split}
&\quad \big\| \int e^{i( \langle x-y, \xi \rangle + t_1 (2 A_1^{-1} \xi'_{\theta}, - \langle \xi'_{\theta}, A_1^{-1} \xi'_{\theta} \rangle) \cdot \xi + t_2 (2 A_2^{-1} \xi'_{\theta}, - \langle \xi'_{\theta}, A_2^{-1} \xi'_{\theta} \rangle ) \cdot \xi)} \beta_{N,k}(\xi) \hat{g}(\xi,t) d\xi \big\|^2_{L^2_{y,t} L^2_{\theta}} \\
&= \int \iint H_{N,k}(t,t') \hat{g}(\xi,t) \hat{\bar{g}}(\xi,t')  dt dt' d\xi
\end{split}
\end{equation*}
\normalsize
with
\small
\begin{equation*}
H_{N,k}(t,t';\xi) = \int e^{i(t_1 - t_1')( 2A_1^{-1} \xi'_{\theta}, - \langle \xi'_{\theta}, A_1^{-1} \xi'_{\theta} \rangle) \cdot \xi + i (t_2 - t_2') ( 2 A_2^{-1} \xi'_{\theta}, - \langle \xi'_{\theta}, A_2^{-1} \xi'_{\theta} \rangle) \cdot \xi} |\beta_{N,k}(\xi)|^2 d \xi'_{\theta}.
\end{equation*}
\normalsize
The behavior of the kernel is again computed by non-stationary phase. First we consider $k \geq 1$. It holds on the support of $\xi'_{\theta}$ for fixed $\xi$:
\begin{equation*}
\nabla_{\xi'_{\theta}} \psi(\xi'_{\theta}) = \{ 2(t_1 - t_1') A_1^{-1} \big( \frac{\xi'}{\xi_3} - \xi'_{\theta} \big) + 2 (t_2 - t_2') A_2^{-1} \big( \frac{\xi'}{\xi_3} - \xi'_{\theta} \big) \} \xi_3,
\end{equation*}
which implies
\begin{equation*}
|\nabla_{\xi'_{\theta}} \psi| \sim N^{\frac{1}{2}} 2^k |t-t'|.
\end{equation*}

We obtain by repeated integration by parts for $M \geq 1$:
\begin{equation*}
|H_{N,k}(t,t';\xi)| \lesssim_M (N^{-\frac{1}{2}} 2^k)^2 (1+ N^{\frac{1}{2}} 2^k |t-t'|)^{-M}.
\end{equation*}
We apply Young's convolution inequality to find
\begin{equation*}
\| A_{\theta}^{N,k} g(y) \|^2_{L^2_{y} L^2_{\theta}} \lesssim N^{-2} \| g \|^2_{L^2_{y}}.
\end{equation*}
Plugging this estimate into \eqref{eq:SobolevEmbeddingMaximalFunction} we find
\begin{equation*}
\| \sup_{\theta} |A_{\theta}^{N,k} g|^2(y) \|_{L^1_y} \lesssim_\varepsilon N^{\varepsilon} \| g \|^2_{L^2_{y,t}}.
\end{equation*}
Summation over $\{ k: 2^k \lesssim N^{\frac{1}{2}} \}$ incurs another logarithmic loss. It remains to note the estimate for $k=0$ given by
\begin{equation*}
|H_{N,k}(t,t';\xi)| \lesssim N^{-1}
\end{equation*}
from disregarding the phase function entirely. The proof of Lemma \ref{lem:MaximalFunctionEstimate} is complete. 

$\hfill \Box$

\section*{Acknowledgements}

Financial support from the Humboldt foundation (Feodor-Lynen fellowship) and partial support by the NSF grant DMS-2054975 is gratefully acknowledged.


\begin{thebibliography}{10}

\bibitem{BiggsBrandesHughes2022}
Kirsti~D. {Biggs}, Julia {Brandes}, and Kevin {Hughes}.
\newblock {Reinforcing a Philosophy: A counting approach to square functions
  over local fields}.
\newblock {\em arXiv e-prints}, page arXiv:2201.09649, January 2022.

\bibitem{Bourgainguth2011}
Jean Bourgain and Larry Guth.
\newblock Bounds on oscillatory integral operators based on multilinear
  estimates.
\newblock {\em Geom. Funct. Anal.}, 21(6):1239--1295, 2011.

\bibitem{ChenGuoYang2023}
Mingfeng {Chen}, Shaoming {Guo}, and Tongou {Yang}.
\newblock {A multi-parameter cinematic curvature}.
\newblock {\em arXiv e-prints}, page arXiv:2306.01606, June 2023.

\bibitem{Cordoba1979}
A.~C\'{o}rdoba.
\newblock A note on {B}ochner-{R}iesz operators.
\newblock {\em Duke Math. J.}, 46(3):505--511, 1979.

\bibitem{Cordoba1982}
Antonio C\'{o}rdoba.
\newblock Geometric {F}ourier analysis.
\newblock {\em Ann. Inst. Fourier (Grenoble)}, 32(3):vii, 215--226, 1982.

\bibitem{Fefferman1973}
Charles Fefferman.
\newblock A note on spherical summation multipliers.
\newblock {\em Israel J. Math.}, 15:44--52, 1973.

\bibitem{GuthMaldague2023}
Larry {Guth} and Dominique {Maldague}.
\newblock {A sharp square function estimate for the moment curve in
  $\mathbb{R}^n$}.
\newblock {\em arXiv e-prints}, page arXiv:2309.13759, September 2023.

\bibitem{GuthWangZhang2020}
Larry Guth, Hong Wang, and Ruixiang Zhang.
\newblock A sharp square function estimate for the cone in {$\Bbb {R}^3$}.
\newblock {\em Ann. of Math. (2)}, 192(2):551--581, 2020.

\bibitem{HickmanZahl2025}
Jonathan {Hickman} and Joshua {Zahl}.
\newblock {Improved $L^p$ bounds for the strong spherical maximal operator}.
\newblock {\em arXiv e-prints}, page arXiv:2502.02795, February 2025.

\bibitem{Maldague2024}
Dominique Maldague.
\newblock A sharp square function estimate for the moment curve in
  {$\Bbb{R}^3$}.
\newblock {\em Invent. Math.}, 238(1):175--246, 2024.

\bibitem{MarlettaRicci1998}
Gianfranco Marletta and Fulvio Ricci.
\newblock Two-parameter maximal functions associated with homogeneous surfaces
  in {$\bold R^n$}.
\newblock {\em Studia Math.}, 130(1):53--65, 1998.

\bibitem{MarlettaRicciZienkiewicz1998}
Gianfranco Marletta, Fulvio Ricci, and Jacek Zienkiewicz.
\newblock Two-parameter maximal functions associated with degenerate
  homogeneous surfaces in {$\bold R^3$}.
\newblock {\em Studia Math.}, 130(1):67--75, 1998.

\bibitem{MockenhauptSeegerSogge1992}
Gerd Mockenhaupt, Andreas Seeger, and Christopher~D. Sogge.
\newblock Wave front sets, local smoothing and {Bourgain}'s circular maximal
  theorem.
\newblock {\em Ann. Math. (2)}, 136(1):207--218, 1992.

\bibitem{Schippa2023Complex}
Robert {Schippa}.
\newblock {Decoupling for complex curves and improved decoupling for the cubic
  moment curve}.
\newblock {\em arXiv e-prints}, page arXiv:2302.10884, February 2023.

\bibitem{Schippa2024}
Robert {Schippa}.
\newblock {Generalized square function estimates for curves and their conical
  extensions}.
\newblock {\em arXiv e-prints}, page arXiv:2408.07248, August 2024.

\bibitem{Sogge2017}
Christopher~D. Sogge.
\newblock {\em Fourier integrals in classical analysis}, volume 210 of {\em
  Camb. Tracts Math.}
\newblock Cambridge: Cambridge University Press, 2nd edition edition, 2017.

\end{thebibliography}
\end{document}